\def\paperversion{2}
\ifnum\paperversion=1
\documentclass[opre,blindrev]{informs4}
\fi
\ifnum\paperversion=2
\documentclass[opre]{informs4}
\fi
\OneAndAHalfSpacedXI %

\usepackage{amsmath,amssymb,amsfonts}
\usepackage{amsmath}
\usepackage{hyperref}
\usepackage{xcolor}
\hypersetup{%
  colorlinks=true,%
  linkcolor={blue},%
  linkbordercolor=red,%
  citecolor={black}
}
\usepackage{endnotes}
\let\footnote=\endnote
\let\enotesize=\normalsize
\def\notesname{Endnotes}%
\def\makeenmark{$^{\theenmark}$}
\def\enoteformat{\rightskip0pt\leftskip0pt\parindent=1.75em
	\leavevmode\llap{\theenmark.\enskip}}

\usepackage{graphicx}
\usepackage{bm}
\usepackage{multirow}
\usepackage{cleveref}

\usepackage[small, margin=1cm]{caption}

\usepackage{color}
\definecolor{strcolor}{rgb}{0.6, 0.2, 0.6}
\definecolor{commentcolor}{rgb}{0.3125, 0.5, 0.3125}
\definecolor{keycol}{rgb}{0, 0, 1}
\usepackage{subfigure}

\usepackage{url}

\DeclareMathOperator{\Diag}{Diag}
\DeclareMathOperator{\tr}{tr}

\renewcommand*{\qed}{\hfill\ensuremath{\square}}

\newtheorem{observation}{Observation}

\renewcommand{\Re}{\mathbb{R}}
\renewcommand{\S}{\mathcal{S}}

\newcommand{\M}{\mathcal{M}}

\def\blot{\quad \mbox{$\vcenter{ \vbox{ \hrule height.4pt
				\hbox{\vrule width.4pt height.9ex \kern.9ex \vrule width.4pt}
				\hrule height.4pt}}$}}

\usepackage{natbib}
\bibpunct[, ]{(}{)}{,}{a}{}{,}%
\def\bibfont{\fontsize{8}{9.5}\selectfont}%

\TheoremsNumberedThrough     %
\ECRepeatTheorems

\EquationsNumberedThrough    %

\gdef\AQ#1{}
\gdef\CQ#1{}
\begin{document}
	
\def\COPYRIGHTHOLDER{INFORMS}%
\def\COPYRIGHTYEAR{2017}%
\def\DOI{\fontsize{7.5}{9.5}\selectfont\sf\bfseries\noindent https://doi.org/10.1287/opre.2017.1714\CQ{Word count = 9740}}

	\RUNAUTHOR{YL.} %

	\RUNTITLE{Adjustable Robust Optimization via Fourier-Motzkin Elimination}

\TITLE{The Augmented Factorization Bound for Maximum-Entropy Sampling}

	\ARTICLEAUTHORS{

\AUTHOR{Yongchun Li}
\AFF{Department of Industrial and Systems Engineering,
University of Tennessee, Knoxville, 37966, TN, USA, ycli@utk.edu}

}

	\ABSTRACT{The maximum-entropy sampling problem (MESP)  aims to select the most informative principal submatrix of a prespecified size from a given covariance matrix. 
  This paper proposes an augmented factorization bound for MESP based on concave relaxation.  
  By leveraging majorization and Schur-concavity theory,
 we demonstrate that this new bound dominates the classic factorization bound of \cite{nikolov2015randomized} and a recent upper bound proposed by \cite{li2024d}. Furthermore, we provide theoretical guarantees that quantify how much our proposed bound improves the two existing ones and establish sufficient conditions for when the improvement is strictly attained. These results allow us to refine the celebrated approximation bounds for the two approximation algorithms of MESP.  Besides, motivated by the strength of this new bound, we develop a variable fixing logic for MESP from a primal perspective. Finally, our numerical experiments demonstrate that our proposed bound achieves smaller integrality gaps and fixes more variables than the tightest bounds in the MESP literature on most benchmark instances, with the improvement 
 being particularly significant when the condition number of the covariance matrix is small.

 }

\KEYWORDS{maximum-entropy sampling, nonlinear integer programming, concave programming, matrix factorization, Schur-concavity}%
\maketitle

\section{Introduction} \label{sec:intro}
The maximum-entropy sampling problem (MESP)  arises in spatial statistics and information theory, which  was introduced by the celebrated work of \cite{shewry1987maximum}. 
MESP aims to select the most informative subset of $s$ variables from a total of 
$n$ variables to maximize the collected information, measured by entropy \citep{ko1995exact}.
It has been widely applied to designing environmental monitoring networks \citep{caselton1984optimal,ozkul2000entropy}.  
When dealing with variables that follow  Gaussian or more general  multivariate elliptical distributions, MESP is recast as a combinatorial optimization problem: 
\begin{align}\label{mesp}
\centering
  z^* := \max_{S} \left\{\log \det \left(\bm C_{S,S}\right): S\subseteq [n], |S| = s \right \}, \tag{MESP}
\end{align}
where $\log\det$ denotes the natural logarithm of the determinant function, $\bm C$ is a sample covariance matrix obtained from the observations of $n$  random variables,
$s\le n$ is a user-specified parameter, and for a subset $S\subseteq [n]$, $\bm C_{S,S}$ denotes a principal submatrix of $\bm C$ indexed by $S$.
We assume that the matrix $\bm C$ is  positive definite, a condition commonly used in the \ref{mesp} literature and well satisfied by the benchmark instances (see \citealt{ko1995exact, sebastiani2000maximum,  anstreicher2018maximum,  anstreicher2020efficient} and references therein).

 \ref{mesp} is computationally challenging and known to be NP-hard, as \cite{ko1995exact} demonstrated. Indeed, \ref{mesp} has no constant-factor polynomial-time approximation algorithm \citep{civril2013exponential}. The primary solution method for solving \ref{mesp} to optimality is branch-and-bound (see the excellent survey by \citealt[chapter 2]{fampa2022maximum} and many references they contain). In practice,  branch-and-cut can quickly find a (near-)optimal solution after only a few iterations; however, proving its optimality can be substantially time-consuming. A tight upper bound of \ref{mesp} is often desired to accelerate this process. 
Inspired by the Schur-concavity of the objective function in \ref{mesp}, as discussed in \Cref{subsec:mono},  this paper develops a tighter upper bound than directly factorizing $\bm C$, a technique commonly used in the literature on \ref{mesp} (see, e.g., \citealt{nikolov2015randomized, chen2023computing,li2024best}). As detailed below,
our method starts with subtracting a scaled identity matrix from $\bm C $ before factorization. 

\begin{remark}\label{remark:decom}
For any $t$, $0\le t\le \lambda_{\min}(\bm C)$,
the matrix $\bm C - t \bm I$ is positive semidefinite, and we denote by $\bm A(t) \in \Re^{n \times n}$ the Cholesky factor of  $\bm C - t \bm I$, i.e., 
\begin{align} \label{eq_decom}
\bm C - t \bm I = \left[\bm A(t)\right]^{\top} \bm A(t). 
\end{align} 
\end{remark}

Note that the rank of the matrix $\bm C-t\bm I$ varies with the value of $t$. Specifically, for $0\le t < \lambda_{\min}(\bm C)$, the matrix is full-rank; however, at $t=\lambda_{\min}(\bm C)$, it becomes singular. The Choleskey factor of a rank-$r$ matrix contains $n-r$ rows of all zeros. These zero rows can be removed to yield a Choleskey factor of size $r\times n$. Thus,  the Cholesky factor is not unique in this context. Fortunately, the bound derived from the matrix factorization is independent of the particular Choleskey factor employed, based on a result of \cite{chen2023computing}. For consistency,  we therefore compute a sized-$n\times n$ Cholesky factor for any $t$.

\subsection{Related work}
The upper bounds of \ref{mesp} have been derived in various ways. An eigenvalue-based upper bound was first introduced by \cite{ko1995exact}. Following this, a variety of eigenvalue-based bounding methods have subsequently been developed and investigated by \cite{anstreicher2004masked,burer2007solving,hoffman2001new,lee2003linear}. 
Another seminal approach to developing upper bounds for \ref{mesp} is based on the continuous relaxations of its equivalent concave integer programs. The classic work of \cite{anstreicher1996continuous,anstreicher1999using} first developed a concave relaxation for \ref{mesp}. Since then, researchers have actively developed different concave integer programs to achieve strong concave relaxations and improve existing methods (see \citealt{nikolov2015randomized,anstreicher2018maximum,anstreicher2020efficient,li2024d}). These bounding methods were further combined and refined by \cite{chen2021mixing,chen2023computing,chen2024masking}.
While no bounding technique wins in all test instances, the ``Linx" bound, as introduced by \cite{anstreicher2020efficient}, and the \ref{eq:fact} bound of \cite{nikolov2015randomized} seem  to provide the tightest upper bounds for \ref{mesp} from a computational perspective. Later, \cite{chen2023computing} applied the known mixing technique to combine them, which can further decrease the bound, especially for values of $s$ close to $n/2$. We refer to the mixing bound as ``Mix-LF" in \Cref{sec:num}. 

Our \ref{eq:upper} bound can be viewed as an augmentation of \ref{eq:fact} of \cite{nikolov2015randomized}  by employing  a general $t$, $0\le t\le \lambda_{\min}(\bm C)$, in \eqref{remark:decom}, leading to a notable reduction in integrality gaps on most test instances. Nevertheless, \cite{nikolov2015randomized} focused on $t=0$ in \eqref{eq_decom} to directly factorize the matrix $\bm C$.
\ref{eq:fact} has been widely recognized for its high effectiveness and computational efficiency.
Hence, 
\ref{eq:fact} and its properties have been extensively studied by \cite{chen2023computing,li2024best,fampa2024outer} and applied to different variants of \ref{mesp}, including generalized \ref{mesp} (\citealt{ponte2024convex}), the D-optimal data fusion (DDF) problem (\citealt{li2024d}),  and the D-optimal design problem (\citealt{ponte2023branch}).
In particular, \cite{li2024d} 
decomposed the matrix $\bm C$ into the form of \eqref{eq_decom} in which they set $t=\lambda_{\min}(\bm C)$, establishing the equivalence between \ref{mesp} and DDF. DDF results from an application to sensor placement  in power systems \citep{li2011phasor}. DDF aims to select a subset of rank-one positive semidefinite matrices to add to the initial Fisher information matrix, so as to maximize the D-optimality of the updated matrix. 
The continuous relaxation of DDF is concave and naturally provides a practical upper bound- \ref{eq:ddf} for \ref{mesp}. 

\subsection{Contributions and Outline}

In \Cref{sec:cip}, we convert \ref{mesp} into a concave integer program for any $t$ using \Cref{remark:decom} and the concave envelope technique, which leads to a new upper bound based on the concave relaxation- \ref{eq:upper}.
    
  \Cref{subsec:mono} highlights the advantages of the proposed factorization method in \Cref{remark:decom}  for improving the upper bound,  through analyzing how the parameter $t$ affects the performance of \ref{eq:upper}. Specifically,
\begin{enumerate}
    \item[(i)] By leveraging  the theory of majorization and Schur-concave functions, we establish that the \ref{eq:upper} bound decreases as $t$ increases in Subsection \ref{subsec:mont_augfact}; 
    \item[(ii)] We demonstrate that for any strictly positive $t$, \ref{eq:upper} is tighter than the two existing upper bounds, \ref{eq:fact} and \ref{eq:ddf}, and provide theoretical guarantees for their difference in optimal values in Subsections \ref{subsec:fact} and \ref{subsec:ddf}, respectively.  Besides, \ref{eq:upper} allows us to enhance the best-known and first-known approximation bounds of the sampling and local search algorithms for \ref{mesp}, respectively; and 
    \item[(iii)] From a primal perspective, Subsection \ref{subsec:fix} develops a variable fixing logic based on a feasible solution of \ref{eq:upper} at $t=\lambda_{\min}(\bm C)$.
\end{enumerate}

The numerical experiments in \Cref{sec:num} demonstrate the superior performance of \ref{eq:upper} across various test instances, compared to most promising bounds in the literature on \ref{mesp}. \Cref{sec:con} summarizes the paper and points to future work.

\noindent\textit{Notations:}  We use bold lower-case letters (e.g., $\bm{x}$) and bold upper-case letters (e.g., $\bm{X}$) to denote vectors and matrices, respectively, and use corresponding non-bold letters (e.g., $x_i, X_{ij}$) to denote their components. We let $\S^n,\S_+^n$ denote the set of all the  $n\times n$ symmetric real matrices and {the} set of all the $n\times n$ symmetric positive semidefinite matrices, respectively.
 We let $\Re^n$ denote the set of all the $n$-dimensional vectors and let $\mathbb{R}_+^n$ denote the set of all the $n$-dimensional nonnegative vectors.
{We let $\bm{1}$ denote the zero vector and let $\bm I$ denote the identity matrix, with their dimensions being clear from the context.}  Given a positive integer $n$ and a positive integer $s\le n$, we let $[n]:=\{1,2,\cdots, n\}$, let $[s,n]:=\{s,s+1,\cdots, n\}$,  and let $\mathbb{I}_s$ be a vector with the first $s$ elements as 1 and the rest as 0. 
For a vector $\bm y \in \Re^n$, we let $y^{\downarrow}_i$ denote the $i$-th largest element of $\bm y$ for each $i\in [n]$, let $\Diag(\bm y)$ denote a diagonal matrix whose diagonal entries consist of the vector $\bm y$, let $\sqrt{\bm y}$ denote a vector where each entry is the square root of that in $\bm y$.
For a  symmetric matrix $\bm{A}$, 
we let $\lambda_{\min}(\bm{X})$ and $\lambda_{\max}(\bm{X})$ denote the smallest and largest eigenvalues of $\bm{X}$, respectively, and let $\bm \lambda(\bm X)$ denote the eigenvalue vector, with eigenvalues sorted in nonincreasing order, that is, $\lambda_i(\bm X)$ is the $i$-th largest eigenvalue of $\bm X$ for each $i\in [n]$. Additional notation will be introduced later as needed.

\section{Reformulating \ref{mesp} as a concave integer program}\label{sec:cip}
In this section, we develop an equivalent concave integer program for \ref{mesp}, based on the proposed  factorization method for matrix $\bm C$ in \Cref{remark:decom} and the concave envelope technique.

\subsection{A naive reformulation of \ref{mesp}}
This subsection presents an equivalent  reformulation of \ref{mesp}, but it is not a concave integer program. 
We begin by introducing the following matrix and function.
\begin{definition}\label{def:matrix}
For any $t$, $0\le t\le \lambda_{\min}(\bm C)$ and a vector $\bm x\in [0,1]^n$, we define a matrix
$$\M_t(\bm x) \in \S_+^n := \sum_{i\in [n]} x_i  \bm a_i(t)  \left[\bm a_i(t)\right]^{\top},$$
where $\bm a_i(t) \in \Re^n$ is the $i$-th column of $\bm A(t)$ for each $i\in [n]$, with $\bm A(t)$ being defined in \Cref{remark:decom}.
\end{definition}
\begin{definition}\label{def:phi}
For a matrix $\bm X \in \S_+^n$ with the eigenvalues $\bm \lambda(\bm X)\in \Re_+^n$, an integer $s\in [n]$, and a constant $t\ge 0$,  we define a function
$$\Phi_s(\bm X; t) :=\sum_{i\in [s]}\log\left(\lambda_{i}(\bm X)+t\right). $$
\end{definition}

Next,  we rewrite \ref{mesp}  as the problem \eqref{eq_mesp} using the function $\Phi_s$. 
When $t=0$, \eqref{eq_mesp} reduces to a known reformulation derived by \cite{li2024best} (see also \citealt{chen2023computing}). We extend this formulation to any $t$ within the range $[0, \lambda_{\min}(\bm C)]$. The most striking result derived from this extension is that as $t$ increases,  the  Lagrangian dual  bound of  \eqref{eq_mesp} becomes tighter, as demonstrated later.

\begin{proposition}\label{prop:mesp}
For any $t$ with $0\le t\le \lambda_{\min}(\bm C)$, \ref{mesp} can be reduced to
\begin{align}\label{eq_mesp}
z^* = \max_{\bm x\in \{0,1\}^n}  \bigg\{\Phi_s\left(\M_t(\bm x); t\right):  \sum_{i\in [n]} x_i = s\bigg\}.
\end{align}
\end{proposition}
\begin{proof}
    For any  $S\subseteq [n]$, $|S|=s$, from \eqref{eq_decom}, we have that
$
    \log\det(\bm C_{S,S}) = \log\det( \left[\left(\bm A(t)\right)_S\right]^{\top} \left(\bm A(t)\right)_S + 
    t\bm I) = \sum_{i\in [s]} \log(\lambda_i + t),
$
where $\lambda_1\ge \cdots \ge \lambda_s\ge 0$ are the eigenvalues of the matrix $\left[\left(\bm A(t)\right)_S\right]^{\top} \left(\bm A(t)\right)_S$.

For a matrix $\bm V \in \Re^{n\times s}$, it is known that $\bm V^{\top}\bm V$ and $\bm V\bm V^{\top}$ have the same first $s$ largest eigenvalues. The matrices $\left[\left(\bm A(t)\right)_S\right]^{\top} \left(\bm A(t)\right)_S$ and $\sum_{i\in S} \bm a_i(t)  \left[\bm a_i(t)\right]^{\top}$ share this property. By \Cref{def:phi}, 
we have that
\begin{align*}
\sum_{i\in [s]} \log(\lambda_i + t) = \Phi_s\bigg(\sum_{i\in S} \bm a_i(t)  \left[\bm a_i(t)\right]^{\top}; t\bigg)=\Phi_s\bigg(\sum_{i\in [n]} x_i\bm a_i(t)  \left[\bm a_i(t)\right]^{\top}; t\bigg)=\Phi_s\left(\M_t(\bm x); t\right),
\end{align*}
where  $\bm x$ is the binary characteristic vector of the subset $S$, i.e., $x_i=1$ if $i\in S$ and $x_i=0$ if $i\in [n]\setminus S$.  \qed
\end{proof}
Unfortunately, the objective function of  \eqref{eq_mesp} is not concave. This motivates us to concavify the  function $\Phi_s$ in the following subsection.

\subsection{The concave envelope}
This subsection presents the concave envelope of  $\Phi_s$, denoted by $\widehat{\Phi}_s$, which allows us to reformulate \eqref{eq_mesp} as a concave integer program.
The \textit{concave envelope} of a  function is the pointwise infimum of all its concave underestimators. First, let us introduce a function.  
\begin{definition}\label{def:psi}
For a vector $\bm y \in \Re_+^n$ and an integer $s\in [n]$, suppose that $k$, $0\le k\le s-1$, is  an integer such that 
 $y^{\downarrow}_{k} > \frac{1}{s-k} \sum_{i\in[k+1,n]} y^{\downarrow}_{i} \ge y^{\downarrow}_{k+1}$, with the convention $y^{\downarrow}_0=\infty$. Then, we define 
  \[\psi_s(\bm y)=\sum_{i\in [k]} \log\left(y^{\downarrow}_i\right) + (s -k) \log \bigg( \frac{\sum_{i\in [k+1, d]} y^{\downarrow}_i}{s-k}\bigg).\]
\end{definition}
Note that the integer $k$ above  is unique, a technical result from \citet[lemma 14]{nikolov2015randomized}. They also established that $\psi_s$ is a concave function with its subgradient defined below.
\begin{remark}[\citealt{nikolov2015randomized}] \label{remark:grad}
Let $\bm y\in \Re_+^n$ be $y_1\ge \cdots \ge y_n\ge 0$, and $k$ follows from \Cref{def:psi}. Then, $\bm g \in \Re^n_+$ is a subgradient of the function $\psi_s$ at $\bm y$, where 
\[g_i=\frac{1}{y_i} ,\ \ \forall i\in [k], \quad g_i=\frac{s-k}{\sum_{i\in [k+1,n]}y_i} ,\ \ \forall i\in [k+1,n]. \] %
\end{remark}

For $t=0$, previous works have established that the concave envelope $\widehat{\Phi}_s(\bm X; 0)$ equals  $\psi_s(\bm \lambda(\bm X))$ for a matrix $\bm X\in \S_+^n$ (see \citealt{nikolov2015randomized, li2024best}). 
They followed the approach of \cite{hiriart1993convex} that computed the \textit{bi-conjugate} of a function to derive its concave envelope. However, applying their proof directly to a general $t$ can be intricate, specifically due to the complexity of solving the two underlying nonlinear optimization problems. It is somewhat surprising that by using the eigen-decomposition technique and perturbing the eigenvalue vector, we  can smoothly extend the established concave envelope result at $t=0$ to  explicitly describe $\widehat{\Phi}_s$ for any $t$, $0\le t\le \lambda_{\min}(\bm C)$.

\begin{proposition}\label{prop:ce}
For any $t$ with $0\le t\le \lambda_{\min}(\bm C)$ and a matrix $\bm X \in \S_+^n$,  the concave envelope of the function $\Phi_s(\bm X; t)$ is 
$ \widehat{\Phi}_s (\bm X; t) = \psi_s(\bm \lambda(\bm X)+ t\mathbb{I}_s).$
\end{proposition}
\begin{proof}
By \Cref{def:phi}, 
we have that $\Phi_s(\bm X; t)= \sum_{i\in [s]} \log\left(\lambda_i(\bm X) +t \right)$ for any matrix $\bm X\in \S_+^n$.  
Suppose that $\bm Q \in \Re^{n\times n}$ are eigenvectors of $\bm X$ corresponding to the eigenvalues $\bm \lambda(\bm X)$. 
It is clear that the eigenvalues of $\bm X + t\bm Q\Diag(\mathbb{I}_s) \bm Q^{\top}$ are $\bm \lambda(\bm X)+ t\mathbb{I}_s$. Then, adding $t\mathbb{I}_s$ does not change the descending order of the entries of ${\bm \lambda}(\bm X)$. By \Cref{def:phi} and the analysis above, we get 
\begin{align*}
  &\Phi_s(\bm X; t)= \sum_{i\in [s]} \log\left(\lambda_i\left(\bm X + t\bm Q\Diag(\mathbb{I}_s) \bm Q^{\top}\right) +0 \right)=\Phi_s\left(\bm X + t\bm Q\Diag(\mathbb{I}_s) \bm Q^{\top}; 0\right)  \\
 \Longrightarrow \ \ &\widehat{\Phi}_s(\bm X; t)=\widehat{\Phi}_s\left(\bm X + t\bm Q\Diag(\mathbb{I}_s) \bm Q^{\top}; 0\right) = \psi_s({\bm \lambda}(\bm X) + t\mathbb{I}_s),
\end{align*}
where the last equation is a result of \cite{nikolov2015randomized}.
\qed
\end{proof}

As a side product,  \Cref{prop:ce} leads to an equivalent concave integer program for \ref{mesp}. 
This paves the way for designing a branch-and-cut algorithm based on (sub)gradient inequalities to solve \ref{mesp} to global optimality (see, e.g., \citealt{li2024best,li2024d}).

\begin{corollary}\label{cor:mesp1}
For any $t$ with $0\le t\le \lambda_{\min}(\bm C)$,
\ref{mesp}  is equivalent to
\begin{align}\label{eq:mesp2}
z^* = \max_{\bm x\in \{0,1\}^n }  \bigg\{ \widehat{\Phi}_s\left(\M_t(\bm x); t\right):  \sum_{i\in [n]} x_i = s\bigg\}. \tag{MESP-I}
\end{align}
\end{corollary}
\begin{proof}
 For a binary vector $\bm x\in \{0, 1\}^{n}$ with $\sum_{i\in [n]} x_i =s$,  suppose $\bm X=\M_t(\bm x)=\sum_{i\in [n]} x_i \bm a_i(t)  \left[\bm a_i(t)\right]^{\top}$.  Then, we observe that $\bm X$ is at most rank-$s$, as $\bm a_i(t)  \left[\bm a_i(t)\right]^{\top}$ is a rank-one  matrix for all $i\in [n]$. Hence, $\bm \lambda(\bm X)$ has  only $s$ positive elements, and so does the perturbed vector $\bm \lambda(\bm X)+t\mathbb{I}_s$. 

Let $\bm y = \bm \lambda(\bm X)+t\mathbb{I}_s$. There is an integer $\ell$, $0\le \ell\le s-1$, such that $y_1\ge \cdots \ge y_{\ell}>  y_{\ell+1} =\cdots = y_s > y_{s+1}=\cdots=y_n=0$, with the convention $y_0=\infty$. Then, it is easy to verify that the integer $k$ in \Cref{def:psi} is exactly $\ell$, and  $\psi_s(\bm y)$ reduces to $\sum_{i\in [s]} \log(y_i)$. 

According to \Cref{prop:ce} and the results above, it follows that
\begin{align}\label{eq:phi}
 \widehat{\Phi}_s(\bm X; t)= \psi_s\left(\bm y\right) = \sum_{i\in [s]} \log(y_i)= \sum_{i\in [s]} \log\left(\lambda_i(\bm X) + t\right) = \Phi_s(\bm X; t), 
\end{align} 
which allows us to replace the objective of \eqref{eq_mesp} with $ \widehat{\Phi}_s$. We thus complete the proof.
\qed
\end{proof}
 
  The continuous relaxation of \ref{eq:mesp2} offers a practical upper bound- \ref{eq:upper}. It is worth noting that \ref{eq:upper} meets the Lagrangian dual bound of  \eqref{eq_mesp}. This is because the concave envelope $\widehat{\Phi}_s$ is precisely the bi-conjugate of the objective function $\Phi_s$ of \eqref{eq_mesp}. By duality and concave conjugate theory, the dual of \eqref{eq_mesp} and the continuous relaxation of \ref{eq:mesp2} form a primal-dual pair.

\section{The augmented factorization bound and its properties}\label{subsec:mono}
Relaxing the binary variables in \ref{eq:mesp2} leads to an upper bound:
\begin{align}\label{eq:upper}
z^*\le \hat{z}(t) := \max_{\bm x\in [0,1]^n }  \bigg\{ \widehat{\Phi}_s\left(\M_t(\bm x); t\right):  \sum_{i\in [n]} x_i = s\bigg\}. \tag{Aug-Fact}
\end{align}
For $t=0$, \ref{eq:upper} reduces to the known factorization bound (\ref{eq:fact}), proposed by \cite{nikolov2015randomized}:
\begin{align}\label{eq:fact}
z^* \le \hat{z}(0) := \max_{\bm x\in [0,1]^n }  \bigg\{\widehat{\Phi}_s\left(\M_0(\bm x); 0\right): \sum_{i\in [n]} x_i = s\bigg\}. \tag{Fact}
\end{align}
In this section, we establish that \ref{eq:upper}  decreases monotonically with $t$, $0\le t\le \lambda_{\min}(\bm C)$. We also demonstrate that \ref{eq:upper} is tighter than the two existing upper bounds- \ref{eq:fact} and \ref{eq:ddf}  for any strictly positive $t$ and quantify how much smaller \ref{eq:upper} is compared to them at $t= \lambda_{\min}(\bm C)$. 
Motivated by the strength of \ref{eq:upper}, we propose a variable fixing logic from a primal perspective.

\subsection{The monotonicity of \ref{eq:upper} and its dominance over \ref{eq:fact}}\label{subsec:mont_augfact}
This subsection investigates the monotonicity of \ref{eq:upper} with respect to  $t$, which allows us to establish that for any $t$ with $0<t \le \lambda_{\min}(\bm C)$, \ref{eq:upper} dominates \ref{eq:fact}. We begin by introducing  \textit{Schur-concave} and \textit{strictly Schur-concave} functions, which are critical to proving our results.
\begin{definition}[\citealt{constantine1983schur,law2007effective}]\label{def:schur} A function $f: \Re^n\to \Re$ is Schur-concave if for all $\bm \nu, \bm \mu \in \Re^n$ such that $\bm \mu $ majorizes  $\bm \nu$ (denoted $\bm \mu \succ \bm \nu$), i.e.,
\[\sum_{i\in [\ell]} \mu_i^{\downarrow} \ge \sum_{i\in [\ell]} \nu^{\downarrow}_{[i]}, \forall \ell \in [n-1], \ \ \sum_{i\in [n]} \mu_{i} = \sum_{i\in [n]} \nu_{i}, \]
one has that $f(\bm \mu) \le f(\bm \nu)$. The function $f$ is strictly Schur-concave if the strict inequality $f(\bm \mu) < f(\bm \nu)$ holds for any $\bm \nu, \bm \mu \in \Re^n$, such that  $\bm \mu \succ \bm \nu$ but $\bm \nu$ is not a permutation of $\bm \mu$. %
\end{definition}

Note that every concave and symmetric function is Schur-concave (see \citealt{marshall1979inequalities}). For the concave function $\psi_s$ \citep{nikolov2015randomized}, we observe that it is permutation-invariant with the arguments.
Therefore, 
\begin{observation}\label{remark:schur}
The function $\psi_s$  in \Cref{def:psi} is Schur-concave. %
\end{observation}

For a matrix $\bm X\in \S_+^n$, by \Cref{prop:ce}, the objective function of \ref{eq:upper} can be represented by $\psi_s$ based on the perturbed eigenvalue vector $\bm \lambda(\bm X)+ t\mathbb{I}_s$. 
The following lemma presents several technical results about the perturbed eigenvalue vector across different $t$. To be specific, Part (i) of \Cref{lem:major}, together with \Cref{remark:schur},  enables us to develop the monotonicity of \ref{eq:upper} in \Cref{them:mont}. Parts (ii) and (iii) facilitate the derivation of a theoretical guarantee for the difference $\hat{z}(0)-\hat{z}(\lambda_{\min}(\bm C))$ in the next subsection.
\begin{lemma}\label{lem:major}
Given a vector $\bm x \in [0,1]^n$ satisfying $\sum_{i\in [n]}x_i=s$, for all $t_1, t_2$ such that $0\le t_1 \le t_2 \le \lambda_{\min}(\bm C)$, suppose that $\bm \nu^{t_1}$ and $\bm \mu^{t_2}$ are the eigenvalues of $\M_{t_1}(\bm x)$ and $\M_{t_2}(\bm x)$, respectively, sorted in nonincreasing order. Then, the following hold:
\begin{enumerate}
    \item[(i)] $ \bm \mu^{t_2} + t_2 \mathbb{I}_s \succ \bm \nu^{t_1} + t_1 \mathbb{I}_s$; 
\item[(ii)] $\sum_{i\in [\ell]}\mu_i^{t_2} + (t_2-t_1) \sum_{i\in [\ell]} x_i^{\downarrow}\ge \sum_{i\in [\ell]}\nu_i^{t_1}$ for each $\ell\in [s]$; and
\item[(iii)] $\mu_i^{t_2} + t_2-t_1 \ge \nu_i^{t_1} $ for each $i\in [s]$.
\end{enumerate}
\end{lemma}
\begin{proof}
The proof of Part (i) is two-step: analyzing the properties of the eigenvalue vectors $\bm \nu^{t_1}, \bm \mu^{t_2}$ and exploring the relation between $\bm \nu^{t_1} + t_1\mathbb{I}_s$ and $\bm \mu^{t_2} +t_2\mathbb{I}_s$, respectively.

\noindent \textbf{Step 1.} For any $t$, let $\bm V=\bm A(t) \Diag(\sqrt{{\bm x}})$. We have that $\Diag(\sqrt{{\bm x}}) \bm A(t)^{\top} \bm A(t) \Diag(\sqrt{{\bm x}}) = \bm V^{\top}\bm V$. On the other hand, it is easy to check that 
$\M_t(\bm x) =\sum_{i\in [n]} x_i \bm a_i(t)  \left[\bm a_i(t)\right]^{\top} = \bm A(t) \Diag(\bm x) \bm A(t)^{\top} = \bm V\bm V^{\top}$. For a matrix $\bm V \in \Re^{n\times n}$, it is known that $\bm V^{\top}\bm V$ and $\bm V\bm V^{\top}$  have the same eigenvalues.  Hence,
the vectors $\bm \nu^{t_1}$ and $\bm \mu^{t_2}$ precisely contain all eigenvalues in the nonincreasing order of the matrices $\Diag(\sqrt{{\bm x}}) \bm A(t_1)^{\top} \bm A(t_1) \Diag(\sqrt{{\bm x}})$ and $\Diag(\sqrt{{\bm x}}) \bm A(t_2)^{\top} \bm A(t_2) \Diag(\sqrt{{\bm x}})$, respectively.

Let $\bm B := \Diag(\sqrt{{\bm x}}) \bm A(t_2)^{\top} \bm A(t_2) \Diag(\sqrt{{\bm x}})$. By  \eqref{eq_decom}, we can get 
\[ \bm C = \bm A(t_1)^{\top} \bm A(t_1) + t_1\bm I = \bm A(t_2)^{\top} \bm A(t_1)+ t_2\bm I \ \ \Longrightarrow \ \ \bm A(t_1)^{\top} \bm A(t_1)  = \bm A(t_2)^{\top} \bm A(t_2)+ (t_2-t_1)\bm I.\]
Multiplying both sides above by $\Diag(\sqrt{{\bm x}})$ gives
\begin{align*}
\Diag(\sqrt{{\bm x}}) \bm A(t_1)^{\top} \bm A(t_1) \Diag(\sqrt{{\bm x}})  = \bm B + (t_2-t_1) \Diag(\bm x),
\end{align*}
which means that  
$\bm \nu^{t_1}$ is the eigenvalue vector of $\bm B + (t_2-t_1) \Diag(\bm x)$.

\noindent\textbf{Step 2.} 
By Ky Fan inequality, for each $\ell \in [n]$, we have that
\begin{equation}\label{eq:kf}
  \begin{aligned}
\sum_{i\in [\ell]} \nu^{t_1}_i &= \sum_{i\in [\ell]} \lambda_i\left(\bm B + (t_2-t_1) \Diag(\bm x)\right) \le \sum_{i\in [\ell]} \lambda_i (\bm B) + (t_2-t_1)\sum_{i\in [\ell]}\lambda_{i} \left( \Diag(\bm x)\right) \\&= \sum_{i\in [\ell]} \lambda_i (\bm B) + (t_2-t_1)\sum_{i\in [\ell]} x^{\downarrow}_i = \sum_{i\in [\ell]} \mu_i^{t_2} + (t_2-t_1)\sum_{i\in [\ell]} x^{\downarrow}_i, 
\end{aligned}  
\end{equation}
where the second equation is because the matrix $\Diag(\bm x)$ is diagonal and its eigenvalues are exactly $\bm x$.

As $\bm x \in [0,1]^n$ and $\sum_{i\in [n]}x_i=s$, we have that $ \sum_{i\in [\ell]} x_{i}^{\downarrow}  \le \min\{\ell, s\}$ for all $\ell\in [n-1]$,
which allows us to further reduce  \eqref{eq:kf} to
\begin{align*}
    \sum_{i\in [\ell]} \nu^{t_1}_i \le     \sum_{i\in [\ell]} \mu^{t_1}_i +(t_2-t_1)\ell, \ \ \forall \ell \in [s], \quad    \sum_{i\in [\ell]} \nu^{t_1}_i \le     \sum_{i\in [\ell]} \mu^{t_1}_i +(t_2-t_1)s, \ \ \forall \ell \in [s+1, n-1],
\end{align*}
and
\begin{align*}
    \sum_{i\in [n]} \nu^{t_1}_i  = \tr\left(\bm B + (t_2-t_1) \Diag(\bm x)\right) = \tr\left( \bm B\right) + (t_2-t_1)s =  \sum_{i\in [n]} \mu^{t_2}_i + (t_2-t_1)s. 
\end{align*}
Hence, we obtain that $\bm \mu^{t_2} + t_2 \mathbb{I}_s$ majorizes $\bm \nu^{t_1} + t_1 \mathbb{I}_s$, i.e., $\bm \mu^{t_2} + t_2 \mathbb{I}_s\succ \bm \nu^{t_1} + t_1 \mathbb{I}_s$.

Part (ii) follows immediately from \eqref{eq:kf}.  

Based on Step 1 of Part (i), we can leverage Weyl's inequality to show Part (iii):
\begin{align*}
\lambda_i (\bm B)\le \lambda_i(\bm B + (t_2-t_1) \Diag(\bm x)) \le  \lambda_i (\bm B) + t_2-t_1  \Longrightarrow \nu^{t_1}_i \le \mu^{t_2}_i + t_2-t_1, \ \ \forall i\in [s].  
\end{align*} 
where the second inequality is due to the fact $\bm x\in [0,1]^n$.
We thus complete the proof. \qed
\end{proof}

\begin{theorem} \label{them:mont}
     \ref{eq:upper}  is monotonically decreasing  with $t$, $0\le t \le \lambda_{\min}(\bm C)$. That is, for all $t_1, t_2$ such that $0\le t_1\le t_2 \le \lambda_{\min}(\bm C)$, the inequality $\hat{z}(t_1)\ge \hat{z}(t_2)$ holds.
\end{theorem}
\begin{proof}
In order to prove the result, we show that for any solution $\bm x\in[0,1]^n$ with $\sum_{i\in [n]}x_i=s$,  the objective function $\widehat{\Phi}_s(\M_t(\bm x); t)$ decreases as $t$ increases. 
Suppose that $\bm \nu^{t_1}$ and $\bm \mu^{t_2}$ are the eigenvalues of $\M_{t_1}(\bm x)$ and $\M_{t_2}(\bm x)$, respectively, sorted in nonincreasing order.  By Part (i) of \Cref{lem:major}, we get
$
\bm \mu^{t_2} + t_2 \mathbb{I}_s \succ    \bm \nu^{t_1}  + t_1 \mathbb{I}_s
$.
According to the Schur-concavity of $\psi_s$ in \Cref{remark:schur}, we have that
$
  \psi_s\left( \bm \nu^{t_1} + t_1 \mathbb{I}_s\right) \ge \psi_s\left(\bm \mu^{t_2} + t_2 \mathbb{I}_s\right)
$. According to \Cref{prop:ce}, the inequality implies that
\begin{align*}
\widehat{\Phi}_s \left(\M_{t_1}(\bm x); t_1\right) \ge  \widehat{\Phi}_s \left(\M_{t_2}(\bm x); t_2\right).
\end{align*}
Thus, it is clear that $\hat{z}(t_1)\ge \hat{z}(t_2)$ holds at optimality.
\qed
\end{proof}

By leveraging the Schur-concavity of the function $\psi_s$, a property not previously explored in the literature, \Cref{them:mont} demonstrates the monotonicity of \ref{eq:upper} over $t$.
To explore this monotonicity, we consider a general $t$, $0\le t \le \lambda_{\min}(\bm C)$ rather than directly setting $t=\lambda_{\min}(\bm C)$ at the beginning.
We also note that

\begin{enumerate}
    \item[(i)]  As  the known \ref{eq:fact} bound is a special case of \ref{eq:upper} at $t=0$, a notable side product of \Cref{them:mont} is that \ref{eq:upper} dominates \ref{eq:fact} whenever $t>0$, as summarized in \Cref{cor:dom}. Our numerical results verify the superior performance of \ref{eq:upper}. In addition, \ref{eq:upper}
  maintains a similar computational efficiency with \ref{eq:fact}, since both objectives are formulated by the concave function $\psi_s$; and
\item[(ii)] 
The proof of \Cref{them:mont} also sheds light on how the objective of \ref{eq:upper} varies with $t$. 
\Cref{cor:mesp1} indicates that the function ${\Phi}_s(\M_{t}(\bm x); t)$ meets its concave envelope $\widehat{\Phi}_s(\M_{t}(\bm x); t)$ if $\bm x$ is a binary solution to \ref{eq:mesp2} (see equation \eqref{eq:phi}). By \Cref{prop:mesp}, the function  ${\Phi}_s(\M_{t}(\bm x); t)$ is invariant under $t$ given a  binary solution $\bm x$, and so is $\widehat{\Phi}_s(\M_{t}(\bm x); t)$.
Interestingly, when  $\bm x$ is not binary, the invariance may not hold, since $\widehat{\Phi}_s(\M_{t}(\bm x); t)$  becomes monotonically decreasing  with $t$.
\end{enumerate}

\begin{theorem}\label{cor:dom}
For any  $t$ with $0< t\le \lambda_{\min}(\bm C)$,   \ref{eq:upper} dominates \ref{eq:fact}, i.e., $\hat{z}(0)\ge \hat{z}(t)$. 
\end{theorem}

\subsection{Theoretical guarantees for the improvement of \ref{eq:upper} over \ref{eq:fact}} \label{subsec:fact}
This subsection aims to quantify the effect of $t$ on \ref{eq:upper}. 
By leveraging \Cref{lem:major} and the concavity of the function $\psi_s$,
we establish a lower bound for the difference $\hat{z}(0)-\hat{z}(\lambda_{\min}(\bm C))$ and propose a sufficient condition where \ref{eq:upper} strictly improves \ref{eq:fact} at $t=\lambda_{\min}(\bm C)$. This lower bound also contributes to enhancing the theoretical performance guarantees of the local search and sampling algorithms for \ref{mesp}.

\begin{theorem}\label{them:bound}
Suppose that $\bm x^*$ is an  optimal solution of \ref{eq:upper} at $t=\lambda_{\min}(\bm C)$ and $\bm \beta^*$ is the vector of eigenvalues of $\M_0(\bm x^*)$ in nonincreasing order.  Then, the following hold:
\begin{enumerate}
 \item[(i)] Let $(x^*)^{\downarrow}_0=0$ by default. We have that
 \begin{align*}
 \hat{z}(0)-  \hat{z}\left(\lambda_{\min}(\bm C)\right) \ge \Delta^{lb} :=\lambda_{\min}(\bm C) \bigg( k-\sum_{i\in [k]}(x^*)^{\downarrow}_i \bigg)\bigg(
\frac{s-k}{\sum_{i\in[k+1,n]}\beta^*_{i}}- \frac{1}{\beta^*_k}\bigg)\ge 0;
 \end{align*}
 and
 \item[(ii)] \ref{eq:upper} with $t=\lambda_{\min}(\bm C)$ strictly dominates \ref{eq:fact} if $k\ge 1$ and $(x^*)^{\downarrow}_k < 1$, 
\end{enumerate}
where $0\le k \le s-1$ is an integer, such that $\beta^*_{k} > \frac{1}{s-k} \sum_{i\in[k+1,n]} \beta^*_{i} \ge \beta^*_{k+1}$ with $\beta^*_0=\infty$.
\end{theorem}
\begin{proof}
Our proof contains two parts.

\begin{enumerate}
    \item[(i)] To begin,
we define $\bm \lambda^*$ to be the eigenvalue vector of $\M_{\lambda_{\min}(\bm C)}({\bm x^*})$. In this way, $\bm \beta^*$ and $\bm \lambda^*$ are a pair of eigenvalues vectors obtained from  $\bm x^*$  at $t=0$ and $t=\lambda_{\min}(\bm C)$, respectively.

Let $\bm \theta^* := \bm \lambda^* + \lambda_{\min}(\bm C)\mathbb{I}_s -\bm \beta^*$. Given the solution $\bm x^*$, by leveraging \Cref{lem:major} in which we set $t_1=0$, $t_2= \lambda_{\min}(\bm C)$ and $\bm \nu^{t_1}=\bm \beta^*$, $\bm \mu^{t_2}=\bm \lambda^*$,  we have that 
\begin{align}\label{ineq}
& \sum_{i\in [n]}\theta^*_i = 0, \quad
 \sum_{i\in [\ell]} \theta^*_i \ge \lambda_{\min}(\bm C) \bigg(\ell-\sum_{i\in [\ell]}(x^*)_{i}^{\downarrow}\bigg), \ \  \forall \ell \in [s], \ \  \text{and} \ \ \theta^*_i\ge 0, \ \ \forall i\in [s].
\end{align}

According to \Cref{prop:ce} and the concavity of $\psi_s$, we have that
\begin{align*}
\hat{z}(\lambda_{\min}(\bm C)) - \hat{z}(0)
&\le \widehat{\Phi}_s (\M_{\lambda_{\min}(\bm C)}(\bm x^*); \lambda_{\min}(\bm C))  - \widehat{\Phi}_s (\M_{0}(\bm x^*); 0)=  \psi_s(\bm \lambda^* +  \lambda_{\min}(\bm C) \mathbb{I}_s) - \psi_s(\bm \beta^*) \\
&\le  \bm g^{\top}(\bm \lambda^*+ \lambda_{\min}(\bm C)\mathbb{I}_s-\bm \beta^*) = \bm g^{\top}\bm \theta^*, 
\end{align*}
where  the first inequality is because   $\bm x^*$ may not be optimal for \ref{eq:fact} and $\bm g \in \Re^n_+ $ is a subgradient of the function $\psi_s$ at $\bm \beta^*$, as defined in \Cref{remark:grad}. Specifically,  $g_i={1}/{\beta^*_i}$ for all $i\in [k]$ and $g_{k+1}=\cdots = g_n= \frac{s-k}{\sum_{i\in [k+1,n]}\beta^*_i}$. By the definition of $\bm g$, we can show that
\begin{align*}
\bm g^{\top}\bm \theta^* & =  \sum_{i\in [k]} g_i\theta^*_i + g_{k+1}   \sum_{i\in [k+1,n]} \theta^*_i \le g_k\sum_{i\in [k]} \theta^*_i +  g_{k+1}   \sum_{i\in [k+1,n]} \theta^*_i  = g_k \sum_{i\in [k]} \theta^*_i -  g_{k+1}   \sum_{i\in [k]} \theta^*_i \\ &
\le (g_{k}-g_{k+1}) \lambda_{\min}(\bm C) \bigg(k-\sum_{i\in [k]}(x^*)^{\downarrow}_i \bigg) = -\Delta^{lb} \le 0
\end{align*}
where the first inequality is from $
\theta^*_i\ge 0$ for all $i\in [k]$ in \eqref{ineq} and $g_1\le \cdots \le g_k$, the second equality is due to the fact that $\sum_{i\in [n]}\theta_i^*=0$ in \eqref{ineq}, 
and the second inequality arises from the lower bound of $\sum_{i\in [k]} \theta^*_i$ in \eqref{ineq}. By definition, we have that $g_k<g_{k+1}$. In addition, the inequality $(k-\sum_{i\in [k]}(x^*)^{\downarrow}_i)\ge 0$ must hold given $\bm x^*\in [0,1]^n$. These results guarantee a nonnegative bound $\Delta^{lb}$.

  \item[(ii)] When $k\ge 1$ and $(x^*)^{\downarrow}_k < 1$, given $x_i^*\le 1$ for all $i\in [n]$, we have that
$k>\sum_{i\in [k]}(x^*)^{\downarrow}_i$. Based on Part (i), it is easy to show that
$
 \hat{z}(0)-  \hat{z}\left(\lambda_{\min}(\bm C)\right) \ge \lambda_{\min}(\bm C)( k-\sum_{i\in [k]}(x^*)^{\downarrow}_i) (g_{k+1}-g_{k}) > 0
$.
 We thus conclude the proof. \qed 
\end{enumerate}
\end{proof}
\Cref{them:bound} provides a theoretical guarantee  $\Delta^{lb}$ for the improvement of   \ref{eq:upper} over \ref{eq:fact}. Part (ii) of \Cref{them:bound} provides a sufficient condition under which \ref{eq:upper} with $t=\lambda_{\min}(\bm C)$ is strictly tighter than \ref{eq:fact}.
A tighter concave relaxation is often beneficial to enhance the theoretical guarantees of approximation algorithms. By leveraging \ref{eq:fact},
\cite{li2024best} derived the best-known and first-known approximation bounds when applying the randomized sampling and local search algorithms to  \ref{mesp}, respectively. We show that \ref{eq:upper} with $t=\lambda_{\min}(\bm C)$ allows us to enhance these approximation bounds by  $\Delta^{lb}$. Analogously, a strict improvement occurs when the condition in Part (ii) of \Cref{them:bound} is satisfied.

\begin{corollary}\label{cor:approx}
The randomized sampling algorithm of \citet[algorithm 2]{li2024best} returns a $(s\log(s/n)+\log(\binom{n}{s}) -\Delta^{lb})$-approximation bound for \ref{mesp}. The local search algorithm of \citet[algorithm 4]{li2024best} returns a   $(s\min\{\log(s), \log(n-s-n/s+2)\} -\Delta^{lb})$-approximation bound for \ref{mesp}.  
\end{corollary}
\begin{proof}
Let $\underline z$ be the objective value of \ref{mesp} returned by the randomized sampling algorithm.  We have that
\begin{align*}
  \underline z &\ge \hat{z}(0) - s\log\left(\frac s n\right ) -\log\left(\binom{n}{s}\right) \ge \hat{z}(\lambda_{\min}(\bm C)) + \Delta^{lb} - s\log\left(\frac s n\right ) -\log\left(\binom{n}{s}\right)\\
  &\ge z^* + \Delta^{lb} - s\log\left(\frac s n\right ) -\log\left(\binom{n}{s}\right),  
\end{align*}
where the first inequality follows from the proof of \citet[theorem 5]{li2024best} and 
the second inequality is because of Part (i) of \Cref{them:bound}. 

For the local search algorithm,  using the result of  \citet[theorem 7]{li2024best}, the rest of the proof follows from the above and is thus omitted.
\qed
\end{proof}

\Cref{them:bound} provides important insights into how the condition number of $\bm C$ affects the performance of \ref{eq:upper} at $t=\lambda_{\min}(\bm C)$. As seen in \Cref{them:bound}, the lower bound $\Delta^{lb}$ is an increasing function of $\lambda_{\min}(\bm C)$. Thus, a larger $\lambda_{\min}(\bm C)$ is desired to guarantee a greater improvement. Besides, the lower bound  is determined by the difference between the reciprocals of the eigenvalues of $\M_0(\bm x^*)$, specifically $(s-k)/(\sum_{i\in[k+1,n]} \beta_i^*) -1/\beta_k^*$. The difference generally decreases as we scale up all the eigenvalues $\bm \beta^*$. It is, therefore, likely that a negative relationship between $\Delta^{lb}$ and $\lambda_{\max}(\bm C)$ exists, given that the eigenvalues of $\M_0(\bm x^*)$ are bounded by  $\lambda_{\max}(\bm C)$ according to \Cref{remark:decom}. Then, a possible implication is that \ref{eq:upper} is more effective at improving \ref{eq:fact} at $t=\lambda_{\min}(\bm C)$ when the condition number of $\bm C$, denoted $\lambda_{\max}(\bm C)/\lambda_{\min}(\bm C)$, is smaller. Our numerical results provide further support for the hypothesis.

\subsection{Theoretical guarantees for the improvement of \ref{eq:upper} over  \ref{eq:ddf}}\label{subsec:ddf}
This subsection generalizes the existing upper bound- \ref{eq:ddf} for \ref{mesp} and demonstrates that \ref{eq:upper} produces a tighter upper bound than \ref{eq:ddf}.

By setting $t=\lambda_{\min}(\bm C)$ in \eqref{eq_decom}, \cite{li2024d} transformed \ref{mesp} into the form of the D-optimality data fusion (DDF) problem. 
We begin with a slight generalization of \citet[theorem 1]{li2024d} to any $t$, $0< t\le \lambda_{\min}(\bm C)$.

\begin{corollary}\label{cor:ddf}
     For any $t$ with $0< t\le \lambda_{\min}(\bm C)$, \ref{mesp} reduces to
     \begin{align}\label{eq:mesp_ddf}
       z^*=  \max_{\bm x\in \{0,1\}^n }  \bigg\{\log\det \left(\M_t(\bm x) + t \bm I \right): \sum_{i\in [n]} x_i=s \bigg\} -(n-s)\log(t).
     \end{align}
\end{corollary}

\begin{proof}
For any subset $S$, $|S|=s$, let $\bm x$ be the binary characteristic vector  of $S$.
Following from the proof of \Cref{prop:mesp}, we get
$
    \log\det(\bm C_{S,S})= \sum_{i\in [s]} \log(\lambda_i + t),
$
where $\lambda_1\ge \cdots \ge \lambda_s\ge 0 = \lambda_{s+1}=\cdots=\lambda_n$ are eigenvalues of the matrix $\M_t(\bm x)$. Given $t>0$, it is easy to check that
\begin{align*}
   \sum_{i\in [s]} \log(\lambda_i + t) = \sum_{i\in [n]} \log(\lambda_i + t) - (n-s)\log(t) = \log\det \left(\M_t(\bm x) + t \bm I \right) -(n-s)\log(t).
\end{align*}
Thus,  we conclude the proof.
\qed
\end{proof}
\Cref{cor:ddf} immediately provides a concave integer program for \ref{mesp}, and it falls into the DDF framework.
 In \eqref{eq:mesp_ddf}, $\M_t(\bm x)$ and $t \bm I$ correspond to the information obtained from newly selected and existing data of DDF, respectively.
A  concave relaxation can be naturally obtained from relaxing the binary variables $\bm x$ of \eqref{eq:mesp_ddf} to be continuous.  We refer to this upper bound as ``DDF-R" to denote the relaxation of DDF. 
\begin{align}\label{eq:ddf}
z^*\le \hat{z}^D(t) := \max_{\bm x\in [0,1]^n }  \bigg\{ \log\det \left(\M_t(\bm x) + t \bm I \right): \sum_{i \in [n]} x_i=s\bigg \}-(n-s)\log(t). \tag{DDF-R}
\end{align}
Note that for $t=0$, \ref{eq:ddf} approaches negative infinity due to the rank deficiency of the objective matrix. Therefore, the condition $0< t\le \lambda_{\min}(\bm C)$ must be satisfied. 

\ref{eq:ddf} has been widely used to provide an upper bound in  branch-and-bound-based methods for finding an optimal solution to DDF (see, e.g., \citealt{hendrych2023solving,li2024d}).
However, \ref{eq:ddf} may only sometimes serve as a strong upper bound, as demonstrated in the numerical results of \cite{li2024d}. They also demonstrated that  \ref{eq:ddf} with $t=\lambda_{\min}(\bm C)$ is not comparable with \ref{eq:fact}. By contrast, our proposed \ref{eq:upper} bound outperforms \ref{eq:ddf}, and it is strictly better in some cases, as shown below.

\begin{theorem}\label{them:ddfdom}
For any $t$ with $0< t\le \lambda_{\min}(\bm C)$, 
the following hold:
\begin{enumerate}
    \item[(i)] \ref{eq:upper} dominates \ref{eq:ddf};
    \item[(ii)]  \ref{eq:upper} strictly dominates \ref{eq:ddf}  if the integer $s$ is strictly less than the rank of  $\bm C - t\bm I$ and  \ref{eq:upper} is not an exact concave relaxation of \ref{mesp}, i.e., $\hat z(t) > z^*$; and
    \item[(iii)] \ref{eq:upper} meets \ref{eq:ddf} if the integer $s$ is no less than the rank of  $\bm C - t\bm I$.
\end{enumerate}
\end{theorem}
\begin{proof}
Our proof contains three parts.
\begin{enumerate}
    \item[(i)] To prove the result, we show that the  objective value of \ref{eq:ddf} is larger than that of \ref{eq:upper} for any feasible solution $\bm x\in [0,1]^n$. Let $\lambda_1\ge \cdots \ge \lambda_n\ge 0$ denote the eigenvalues of $\M_t(\bm x)$.  Suppose $0\le k \le s-1$ is an integer, such that $\lambda_{k} > \frac{1}{s-k} \sum_{i\in[k+1,n]} \lambda_{i} \ge \lambda_{k+1}$, with the convention $\lambda_0=\infty$. Then, we construct a vector $\bm \beta\in \Re_+^n$ as
\[ \beta_i = \lambda_{i}, \forall i\in [k], \ \ \beta_{k+1} =\cdots=\beta_s = \frac{1}{s-k} \sum_{j\in[k+1,n]} \lambda_{j}, \ \ \beta_i = 0,\forall i\in[s+1, n].\]
From the construction above, we get
\begin{align*}
 \sum_{i\in [n]}\log(\beta_i+ t)- (n-s)\log(t) = \sum_{i\in [s]}\log(\beta_i+ t)  = \psi_s(\bm \lambda+ t\mathbb{I}_s) = \widehat{\Phi}_s\left(\M_t(\bm x); t\right),
\end{align*}
where the first equation is due to  $\beta_i=0$ for all $i\in [s+1, n]$, the second one is from \Cref{def:psi}, and the last one is from \Cref{prop:ce}.

In addition, it is easy to verify that $\bm \beta \succ \bm \lambda$.   Majorization remains valid after adding the vector $t\bm 1$; that is,   $\bm \beta + t\bm 1 \succ \bm \lambda+ t\bm 1$. It is known that for a vector $\bm y\in \Re_{++}^n$, the function $\sum_{i\in [n]}\log(y_i)$ is strictly Schur-concave (see, e.g., \citealt{marshall1979inequalities,shi2007schur}). Thus, the objective value of \ref{eq:ddf} satisfies
\begin{align*}
     \sum_{i\in [n]}\log(\lambda_i+ t) & \ge     \sum_{i\in [n]}\log(\beta_i+ t) =  \widehat{\Phi}_s\left(\M_t(\bm x); t\right) + (n-s)\log(t).
\end{align*}
Thus, we must have $\hat{z}^D(t)\ge \hat z(t)$ at optimality.
\item[(ii)] For any  $t$ with $0< t\le \lambda_{\min}(\bm C)$, suppose $\bm x^*$ is an optimal solution to \ref{eq:upper}. Then, $\bm x$ must not be binary.  Otherwise, $\bm x^*$ is also optimal for \ref{eq:mesp2}, which contradicts with $\hat{z}(t)>z^*$. Thus, the support of $\bm x^*$ is at least size-$(s+1)$. In addition, the rank of $\bm C -t\bm I$ is strictly greater than $s$. By the definition of $\M_t(\bm x^*)$ and \Cref{remark:decom},
its rank must exceed $s$ in this context. 

Let $\lambda_1\ge \cdots \ge \lambda_n\ge 0$ denote the eigenvalues of $\M_t(\bm x^*)$. Then, $\lambda_{s+1}$ is strictly positive. Following Part (i), we construct a vector $\bm \beta$ with $\beta_{s+1}=0$. Given $\lambda_{s+1}>\beta_{s+1} $, $\bm \lambda+t\bm 1$ can not be a permutation of $\bm \beta+t \bm 1$. Following Part (i) to use the property of a strictly Schur-concave function, we obtain that
\begin{align*}
     \sum_{i\in [n]}\log(\lambda_i+ t) & >   \sum_{i\in [n]}\log(\beta_i+ t) =  \widehat{\Phi}_s\left(\M_t(\bm x^*); t\right) + (n-s)\log(t) = \hat{z}(t) + (n-s)\log(t).
\end{align*}
As $\bm x^*$ is feasible for \ref{eq:ddf}, the optimal value $\hat{z}^D(t)$ must be strictly greater than $\hat{z}(t)$.
\item[(iii)] We establish that the objective values of \ref{eq:ddf} and \ref{eq:upper} are equal in this case for any feasible solution $\bm x\in [0,1]^n$. It suffices to prove that the vectors $\bm \lambda $ and $\bm \beta$ in Part (i) are the same.
Let $r$ be the rank of $\bm C -t\bm I$. Given $r\le s$ and $\sum_{i\in [n]}x_i=s$, according to \Cref{remark:decom} and \Cref{def:matrix}, the matrix $\M_t(\bm x)$ must be rank-$r$, and thus, its eigenvalues satisfy $\lambda_1\ge \cdots \ge  \lambda_r >\lambda_{r+1}=\cdots=\lambda_{n}=0$.  Next, there are two cases to be discussed.
\begin{enumerate}
    \item $r=s$. First, there always exists an integer $0\le \ell\le s-1$  such that $\lambda_{\ell} > \lambda_{\ell+1}=\cdots=\lambda_r $, with the convention $\lambda_{0}=\infty$. We can verify that $\lambda_{\ell} > \frac{1}{s-\ell} \sum_{i\in [\ell+1, n]} \lambda_i=\frac{1}{r-\ell} \sum_{i\in [\ell+1, r]} \lambda_i = \lambda_{\ell+1}$, where the first equation follows from the facts that $s=r$ and $\lambda_{r+1}=\cdots=\lambda_n=0$. The integer $k$ in Part (i) is unique, and thus, it must equal $\ell$. By the construction of $\bm \beta$, we have that $\bm \beta = \bm \lambda$. 
    
    \item $r<s$. It is clear that $\lambda_r > 0 =\frac{1}{s-r}\sum_{i\in [r+1, n]}\lambda_i=\lambda_{r+1}$. Here, the integer $k$ in Part (i) equals $r$. It follows that that $\bm \beta = \bm \lambda$. 
\end{enumerate}
Since the objective values of \ref{eq:ddf} and \ref{eq:upper} are always equal, their optimal values must be the same. We thus complete the proof.
\qed
\end{enumerate}
\end{proof}
We would like to highlight that 
both conditions in Part (ii) of \Cref{them:ddfdom} can be readily satisfied. That is, \ref{eq:upper} strictly dominates \ref{eq:ddf} in most cases.  First,  when $0<t<\lambda_{\min}(\bm C)$, the matrix $\bm C-t\bm I$ is full-rank. Thus, the first condition is, in fact,  the inequality $s\le n-1$ under this setting.  When $s=n$, it is the trivial case, as both \ref{eq:upper} and \ref{eq:ddf} yield the same optimal values as \ref{mesp}. For $t=\lambda_{\min}(\bm C)$, the matrix $\bm C-t\bm I$ has a rank at most  $n-1$. We use $t=\lambda_{\min}(\bm C)$ in the numerical study, where the first condition reduces to $s\le n-2$.
Second, if \ref{eq:upper} matches \ref{mesp}, i.e., $\hat z(t) = z^*$, it is undoubtedly the strongest upper bound. 

As \ref{eq:upper} is stronger than \ref{eq:fact} for any $t$ with $0<t\le \lambda_{\min}(\bm C)$, Part (iii) of \Cref{them:ddfdom} results in a sufficient condition under which \ref{eq:ddf} dominates \ref{eq:fact}.
\begin{corollary}
Suppose that the integer $s$ is no less than the rank of $\bm C-t\bm I$. Then, \ref{eq:ddf} dominates \ref{eq:fact} for any $t$, $0<t\le \lambda_{\min}(\bm C)$. 
\end{corollary}

Analogous to \ref{eq:upper},  we show that \ref{eq:ddf} decreases monotonically as $t$ increases by leveraging the theory of Schur-concavity. This indicates that setting $t=\lambda_{\min}(\bm C)$ yields the best \ref{eq:ddf} bound, which is exactly the one proposed by \cite{li2024d}.

\begin{proposition}
\label{them:montddf}
     \ref{eq:ddf}  is monotonically decreasing  with $t$, $0< t \le \lambda_{\min}(\bm C)$. That is, for all $t_1, t_2$ such that $0 < t_1\le t_2 \le \lambda_{\min}(\bm C)$, the inequality $\hat{z}^D(t_1)\ge \hat{z}^D(t_2)$ holds.
\end{proposition}
\begin{proof}
In the following, we show that given a feasible solution $\bm x$ of \ref{eq:ddf},  the objective function is monotonically decreasing  with $t$.
Suppose that $\bm \nu^{t_1} \in \Re_+^n$ and $\bm \mu^{t_2}\in \Re_+^n$ are the vectors of the eigenvalues of $\M_{t_1}(\bm x)$ and $\M_{t_2}(\bm x)$, respectively, sorted in nonincreasing order.

Part (i) of \Cref{lem:major} implies that $\bm \mu^{t_2} + (t_2-t_1)\mathbb{I}_s \succ \bm \nu^{t_1}$. Adding  $t_1 \bm{1}$ on both sides directly leads to $\bm \mu^{t_2} + (t_2-t_1)\mathbb{I}_s +t_1\bm{1} \succ \bm \nu^{t_1}+t_1\bm{1}$. As the function $\sum_{i\in [n]}\log(y_i)$ is Schur-concave for $\bm y\in \Re_{++}^n$ (see, e.g., \citealt{marshall1979inequalities}), we have that
\begin{align*}
&\log\det(\M_{t_1}(\bm x) + t_1\bm I)  = \sum_{i\in [n]}\log(\nu_i^{t_1}+ t_1) \ge     \sum_{i\in [s]}\log(\mu_i^{t_2}+ t_2)+ \sum_{i\in [s+1, n]}\log(\mu_i^{t_1}+ t_1) \\
&\ge  \sum_{i\in [n]}\log(\mu_i^{t_2}+ t_2)
- (n-s)\log\left(\frac{t_2}{t_1}\right) = \log\det(\M_{t_2}(\bm x) + t_2\bm I) - (n-s)\log\left(\frac{t_2}{t_1}\right),
\end{align*}
where the second inequality is because  $\log(\mu_i^{t_2}+ t_1)+\log(t_2/t_1)= \log(t_2/t_1\mu_i^{t_2}+ t_2)\ge\log(\mu_i^{t_2}+ t_2)$ for all $i\in [s+1, n]$. Thus, the monotonicity of \ref{eq:ddf} immediately stems from its  monotonic objective over $t$. We conclude the proof. \qed
\end{proof}

Next, we derive a theoretical bound for the difference in optimal values between \ref{eq:upper} and \ref{eq:ddf} with $t=\lambda_{\min}(\bm C)$ using the property of the natural logarithmic function.
\begin{theorem}\label{them:ddf}
Suppose that $\bm x^*$ is an optimal solution of  \ref{eq:upper} $t=\lambda_{\min}(\bm C)$ and the vector $\bm \lambda^* \in \Re_+^n$ contains the eigenvalues of $\M_{\lambda_{\min}(\bm C)}(\bm x^*)$ in nonincreasing order. Then, we have that
\[\hat{z}^D(\lambda_{\min}(\bm C)) - \hat{z}(\lambda_{\min}(\bm C)) \ge \Theta^{lb} := \left(\frac{1}{\lambda^*_{s+1} + \lambda_{\min}(\bm C)}-\frac{1}{\lambda^*_{s}+\lambda_{\min}(\bm C)} \right) \sum_{i\in [s+1,n]} \lambda^*_i  \ge 0.\]
\end{theorem}
\begin{proof}
First, we construct a vector $\bm \beta\in \Re^n_+$ as
\[ \beta_i = \lambda^*_{i}, \forall i\in [k], \ \ \beta_{k+1}=\cdots=\beta_s = \frac{1}{s-k} \sum_{j\in[k+1,n]} \lambda^*_{j},  \ \ \beta_i = 0,\forall i\in[s+1, n],\]
where the integer $0\le k\le s-1$ satisfies $\lambda^*_{k} > \frac{1}{s-k} \sum_{i\in[k+1,n]} \lambda^*_{i} \ge \lambda^*_{k+1}$. Then, we have that
\begin{align*}
 \hat{z}^D(t) - \hat{z}(t) \ge \sum_{i\in [n]}\log(\lambda^*_i+ t)-\sum_{i\in [n]}\log(\beta_i+ t) = \sum_{i\in [k+1,n]}
 \log\left(\frac{\lambda^*_i+ t}{\beta_i+t}\right) \ge \sum_{i\in [k+1,n]}  \left(1-\frac{\beta_i+ t}{\lambda^*_i+t}\right),
\end{align*}
where the first inequality is because $\bm x^*$ is feasible for \ref{eq:ddf}, the first equation is from \Cref{prop:ce} that implies $\log\det (\M_t(\bm x^*))=\psi_s(\bm \lambda^* + t \mathbb{I}_s)$, the second equation is because $\lambda_i=\beta_i$ for all $i\in [k]$, and the last inequality stems from the fact that for any $y>0$, $\log(y) \ge 1-1/y$ must hold. 

Next, we show that  the right-hand expression above is bounded by
\begin{align*}
\sum_{i\in [k+1,n]} \frac{\lambda^*_i-\beta_i}{\lambda^*_i+t}& = \sum_{i\in [k+1,s]}  \frac{\lambda^*_i-\beta_i}{\lambda^*_{i}+t} + \sum_{i\in [s+1,n]} \frac{\lambda^*_i}{\lambda^*_i+t} \ge \sum_{i\in [k+1,s]}  \frac{\lambda^*_i-\beta_i}{\lambda^*_{s}+t} + \sum_{i\in [s+1,n]} \frac{\lambda^*_i}{\lambda^*_{s+1}+t} \\
&\ge\left(\frac{1}{\lambda^*_{s+1}+t}-\frac{1}{\lambda^*_{s}+t} \right) \sum_{i\in [s+1,n]} \lambda^*_i \ge 0,
\end{align*}
where the first equation is by the definition of $\bm \beta$ and the inequalities stem from the facts that $\beta_i = \frac{1}{s-k} \sum_{j\in[k+1,n]} \lambda^*_{j} \ge \lambda^*_i$ for all $i\in [k+1, s]$, $\sum_{i\in [k+1,s]} (\lambda^*_i -\beta_i) + \sum_{i\in [s+1,n]} \lambda^*_i=0$, and $\lambda^*_{s+1}\le \lambda^*_{s}$.
We thus conclude the proof. 
\qed
\end{proof}

We close this subsection by discussing how the lower bound $\Theta^{lb}$ in \Cref{them:ddf} varies with the condition number of $\bm C$.
 Following \Cref{them:bound}, we note  that \ref{eq:upper} with $t=\lambda_{\min}(\bm C)$ may be more effective  at improving \ref{eq:fact}  when the condition number of $\bm C$ is small. Conversely, \Cref{them:ddf} suggests that {the improvement of \ref{eq:upper}  over \ref{eq:ddf} becomes notable} given a large condition number, as detailed below. Our numerical studies also demonstrate that \ref{eq:upper} is significantly tighter than \ref{eq:ddf} when its improvement over \ref{eq:fact} is minor, and vice versa. 
  
As seen in \Cref{them:ddf}, the lower bound $\Theta^{lb}$ decreases as $\lambda_{\min}(\bm C)$ increases. Besides, the  bound $\Theta^{lb}$ generally increases as we scale up all the eigenvalues $\bm \lambda^*$ of $\M_t(\bm x^*)$. Note that these eigenvalues  are bounded by  $\lambda_{\max}(\bm C)$ based on \Cref{remark:decom}. It is possible, therefore, that there exists a positive relationship between $\Theta^{lb}$ and $\lambda_{\max}(\bm C)/\lambda_{\min}(\bm C)$. Thus,  a large condition number of $\bm C$ is desirable for achieving  a notable improvement of \ref{eq:upper}  over \ref{eq:ddf} at $t=\lambda_{\min}(\bm C)$, as shown in \Cref{sec:num}.

\subsection{A primal certificate for variable fixing using \ref{eq:upper}}\label{subsec:fix}
Variable fixing has been extensively
studied for \ref{mesp} in the context of various concave relaxation bounds (see, e.g., \citealt{anstreicher2001maximum,anstreicher2018maximum,anstreicher2020efficient,chen2023computing} and references therein). It is often used to accelerate the computation of exact solution methods \citep{li2024d}.
However, previous research has focused on deriving dual certificates, which requires computing  (near-)optimal dual solutions of those concave relaxations. 
By contrast, this subsection introduces a primal certificate for variable fixing using the property of concave functions, independent of the dual problem of \ref{eq:upper}.

To begin, we need an expression for the subgradient of the objective function of \ref{eq:upper}. Note that  $\widehat{\Phi}_s(\M_t(\bm x); t)$ is a spectral function that only depends on the eigenvalues of $\M_t(\bm x)$. Based on the spectral property and the subgradient of $\psi_s$ in \Cref{remark:grad}, \cite{li2024best} derived the subgradient of $\widehat{\Phi}_s(\M_t(\bm x); t)$ over $\bm x$ at $t=0$ (see also \citealt{chen2023computing}). Their result can directly extend to any $t$, $0\le t\le \lambda_{\min}(\bm C)$.
\begin{remark}\label{remark:subgrad}
For any feasible solution $\bm x$ of \ref{eq:upper} and any $t$, $0\le t\le \lambda_{\min}(\bm C)$,  suppose that $\M_t(\bm x)=\bm Q \Diag(\bm \lambda) \bm Q^{\top}$ is the eigen-decomposition of $\M_t(\bm x)$, where $\lambda_1\ge\cdots\ge \lambda_n\ge 0$ are the eigenvalues of $\M_t(\bm x)$. Then, a subgradient of $\widehat{\Phi}_s(\M_t(\bm x); t)$ at $\bm x$ can be defined as 
\begin{align*}
\frac{\partial \widehat{\Phi}_s\left(\M_t({\bm x}); t\right) }{\partial x_i} = [\bm a_i(t)]^{\top} \bm Q \Diag(\bm g) \bm Q^{\top} \bm a_i(t), \ \ \forall i\in [n],
\end{align*}
where $\bm g$ is a subgradient of the function $\psi_s$ at $\bm \lambda$, as defined in \Cref{remark:grad}. %
\end{remark}

\begin{theorem}\label{them:fix}
For any feasible solution $\tilde{\bm x}$  of \ref{eq:upper} at $t=\lambda_{\min}(\bm C)$, let $\tilde{\bm g}$ be a subgradient of the function $\widehat{\Phi}_s$ at $\tilde{\bm x}$, as defined in \Cref{remark:subgrad}. Then, any optimal solution $\bm x^*$ of \ref{mesp} must satisfy
\begin{align*}
   & x_i^* = 1 \ \ \text{if} \ \ \tilde g_i - \tilde g_{s+1}^{\downarrow} > UB - LB, \ \ \forall i\in [n], \text{ and } \\
   & x_i^*=0 \ \ \text{if} \ \ \tilde g_{s}^{\downarrow} - \tilde g_i > UB - LB, \ \ \forall i\in [n],
\end{align*}
where $UB  = \widehat{\Phi}_s\left(\M_{\lambda_{\min}(\bm C)} (\tilde{\bm x}); \lambda_{\min}(\bm C)\right) - \tilde{\bm g}^{\top} \tilde{\bm x} + \sum_{i\in [s]} \tilde g^{\downarrow}_i$ and $LB$ is a lower bound of \ref{mesp} returned by approximation algorithms.
\end{theorem}

\begin{proof}
For $t=\lambda_{\min}(\bm C)$, by the concavity of  $\widehat{\Phi}_s$, we have that
\begin{align*}
\widehat{\Phi}_s(\M_t({\bm x}); t) \le \widehat{\Phi}_s(\M_t(\tilde{\bm x}); t) + \tilde{\bm g}^{\top} (\bm x- \tilde{\bm x}) 
\end{align*}
for all $\bm x\in[0,1]^n$ with the cardinality $s$. Maximizing the above inequality over $\bm x$ results in
\begin{equation}\label{ineq2}
    \begin{aligned}
        \hat{z}(t) &= \max_{\bm x\in [0,1]^n} \bigg\{\widehat{\Phi}_s(\M_t({\bm x}); t):\sum_{i\in [n]} x_i=s \bigg\} \le \widehat{\Phi}_s(\M_t(\tilde{\bm x}); t) -  \tilde{\bm g}^{\top}  \tilde{\bm x}  + \max_{\bm x\in [0,1]^n}\bigg\{ \tilde{\bm g}^{\top} \bm x :\sum_{i\in [n]} x_i=s \bigg\} \\
        &= \widehat{\Phi}_s(\M_t(\tilde{\bm x}); t) - \tilde{\bm g}^{\top} \tilde{\bm x} + \max_{\bm x\in \{0,1\}^n}\bigg\{\tilde{\bm g}^{\top} \bm x:\sum_{i\in [n]} x_i=s \bigg\} = UB,
    \end{aligned}
\end{equation}
where the second equation follows from the linearity of the objective and the last equation is because the maximization problem attains the optimal value $\sum_{i\in [s]} \tilde g_i^{\downarrow}$. 

We split the following proof into two parts, which fix a variable to 1 and 0, respectively.
\begin{enumerate}
    \item[(i)]  For each $i\in [n]$, we assume that $x_i=0$. If \ref{eq:mesp2} strictly decreases when restricted to satisfying the constraint $x_i=0$, then no optimal solution of \ref{eq:mesp2} can satisfy $x_i=0$. Therefore, $x_i$ must be equal to 1 at optimality. Next, our goal is to provide a sufficient condition under which \ref{eq:mesp2} with $x_i=0$ is strictly less than \ref{eq:mesp2}.  
Suppose that $\hat{z}_i^0(t)$ denotes the optimal value \ref{eq:upper} with the constraint $x_i=0$. Following the analysis in \eqref{ineq2}, we have that
\begin{align*}
\hat{z}_i^0(t) \le \widehat{\Phi}_s(\M_t(\tilde{\bm x}); t) -  \tilde{\bm g}^{\top}\tilde{\bm x}  + \max_{\bm x\in \{0,1\}^n}\bigg\{ \tilde{\bm g}^{\top} \bm x:\sum_{j\in [n]} x_j=s, x_i=0 \bigg\}.
\end{align*}
Enforcing the constraint $x_i=0$ leads to
\begin{align*}
\max_{\bm x\in \{0,1\}^n}\bigg\{ \tilde{\bm g}^{\top} \bm x:\sum_{j\in [n]} x_j=s, x_i=0 \bigg\} = \begin{cases}
   \sum_{j\in [s]} \tilde g_j^{\downarrow}, & \text{if } \tilde g_i\le \tilde g^{\downarrow}_{s+1}; \\
    \sum_{j\in [s]} \tilde g_j^{\downarrow} + \tilde g^{\downarrow}_{s+1}- \tilde g_i , & \text{if } \tilde g_i\ge \tilde g^{\downarrow}_{s}.
\end{cases}.
\end{align*}

If $UB + \tilde g_{s+1}- \tilde g_i< LB$ holds, we have that $\hat{z}_i^0(t) < LB\le z^*$ based on the results above.
Given that $\hat{z}_i^0(t)$ serves as an upper bound of \ref{eq:mesp2} with the constraint $x_i=0$,  $x_i$ must be equal to 1 at optimality of \ref{eq:mesp2}, as analyzed previously.

\item[(ii)] 

For each $i\in [n]$, suppose $x_i=1$ in \ref{mesp}. Then, we have that
\begin{align*}
   \max_{\bm x\in \{0,1\}^n}\bigg\{ \tilde{\bm g}^{\top} \bm x:\sum_{j\in [n]} x_j=s, x_i=1 \bigg\}    = \begin{cases}
   \sum_{j\in [s]} \tilde g_j^{\downarrow}, & \text{if } \tilde g_i\ge \tilde g^{\downarrow}_{s}; \\
    \sum_{j\in [s]} \tilde g_j^{\downarrow} - \tilde g^{\downarrow}_{s}+\tilde g_i , & \text{if } \tilde g_i\le \tilde g^{\downarrow}_{s+1}.
\end{cases}.
\end{align*}
The rest of the proof simply follows that of Part (i) and is thus omitted. \qed
\end{enumerate}
\end{proof}
We make the following remarks about \Cref{them:fix}.
\begin{enumerate}
    \item[(i)] The main advantage of our dual-free variable fixing in \Cref{them:fix} is its ease of implementation- it can be easily integrated into any first-order algorithm for solving \ref{eq:upper}. At each iteration, to fix variables, it suffices to sort the elements of the subgradient; 
    \item[(ii)] Our variable fixing conditions in \Cref{them:fix}  align well with the cardinality constraint in \ref{mesp}. Given that $UB-LB\ge 0$ and the subgradient vector has at most $s$ entries larger than its $s+1$ largest entry, we can fix up to $s$ variables to 1. Likewise, we can fix at most $n-s$ variables to  0; and
    \item[(iii)] Our theoretical analysis of \Cref{them:fix} builds on a feasible solution \ref{eq:upper}. In fact, it can be directly generalized to other upper bounds based on concave relaxations. For example, we evaluate the  variable-fixing capacity of \ref{eq:fact} and \ref{eq:ddf} based on our primal certificate in Subsection \ref{subsec:com_ddf}.
\end{enumerate}

\section{Numerical experiments}\label{sec:num} 
In this section, we numerically compare \ref{eq:upper} with the existing upper bounds of \ref{mesp} and verify its dominance over \ref{eq:fact} and \ref{eq:ddf} with varying-scale instances. As  both \ref{eq:upper} and \ref{eq:ddf} are decreasing with $t$, we set $t=\lambda_{\min}(\bm C)$ for them throughout this section. Besides, we use the Frank-Wolfe algorithm to compute the upper bounds. To obtain a high-quality lower bound of \ref{mesp}, we employ the local search algorithm proposed by \cite{li2024best} that has returned an optimal solution to \ref{mesp} on three benchmark data sets. 
All the experiments
are conducted in Python 3.6 with calls to Gurobi 9.5.2 and MOSEK 10.0.29 on a PC with 10-core CPU, 16-core GPU, and 16GB of memory. 

\subsection{MESP: Three benchmark data sets}\label{subsec:mesp}
To evaluate \ref{eq:upper}, we first consider three benchmark covariance matrices with $n=63, 90, 124$. Their corresponding condition numbers are 48.42, 200.45, and 78340.48, respectively. Both  $n=63$ and $n=124$ instances have been repeatedly used in the literature on \ref{mesp}, which are collected from an application to re-designing an environmental monitoring network \citep{guttorp1993using}. Recently, \cite{anstreicher2020efficient} considered the $n=90$ instance for \ref{mesp} based on temperature data from monitoring stations in the Pacific Northwest of the United States. \Cref{fig_comp_upper,fig_comp_upper_90,fig_comp_upper_124} display the gaps between several upper bounds and a lower bound generated by the local search algorithm. We note that the gap values for Fact, Linx, and Mix-LF are taken from the computational results of \cite{chen2023computing}.
For each benchmark instance, gap values are given for $s\in [2,n-1]$. 
Their computational time is negligible (i.e., less than one minute), so we do not report and compare them.

 \Cref{63_gap,90_gap} show that \ref{eq:upper} gives the best upper bound for \ref{mesp} on the $n=63$ and $n=90$ data sets.  Surprisingly, our \ref{eq:upper} reduces the integrality gaps effectively for the most difficult instances, with intermediate values of $s$,
where \ref{eq:fact} and Linx are nearly identical. Consistent with our analysis of \Cref{them:bound}, \ref{eq:upper} is only a bit better than \ref{eq:fact} for $n=124$, as displayed in \Cref{124_gap}.
This is because the $n=124$ covariance matrix has a huge condition number. We present the comparison between \ref{eq:upper} and \ref{eq:ddf} in different figures, since \ref{eq:ddf} often results in much worse integrality gaps. We see from \Cref{63_gap_ddf,90_gap_ddf,124_gap_ddf} that \ref{eq:upper} is much tighter than \ref{eq:ddf} for all the test cases. As $s$ approaches $n-1$, the gaps produced by \ref{eq:upper} and \ref{eq:ddf} become nearly identical. It is interesting to observe that the difference of gaps between \ref{eq:upper} and \ref{eq:ddf} is increasing with the condition number of $\bm C$. Especially for $n=124$, we observe a significant reduction in the gaps in \Cref{124_gap_ddf}.
These comparison results parallel our theoretical findings in Subsection \ref{subsec:ddf}.

We verify the enhanced capacity of \ref{eq:upper} to fix variables for \ref{mesp} in \Cref{63_fix,90_fix,124_fix}, when compared to \ref{eq:upper} and \ref{eq:ddf}. Note that we employ the primal conditions in \Cref{them:fix} to check whether to fix a variable. For all other bounds, their fixed variables are sourced from \citet[section 3]{chen2023computing}, using the dual certificates.
For  $n=64$ and $n=90$, we see that \ref{eq:upper} fixes many variables for large values of $s$, while \ref{eq:fact} fails to fix any variables at all. For $n=124$, \ref{eq:upper} still leads to more variables fixed than \ref{eq:fact}, even at points in which they have very similar gaps. In addition, we observe a slightly different comparison result between  \ref{eq:upper} and Linx for $n=90$ and $n=124$.  Specifically,  for small values of $s$, \ref{eq:upper} has a smaller integrality gap and a stronger fixing power  than the Linx bound, whereas the reverse holds when $s$ is large.

\vspace{-0.5em}
\begin{figure}[ht]
	\centering
 	\hspace{-1em}
	\subfigure[Gaps ] {\label{63_gap}
	\includegraphics[width=0.31\textwidth]{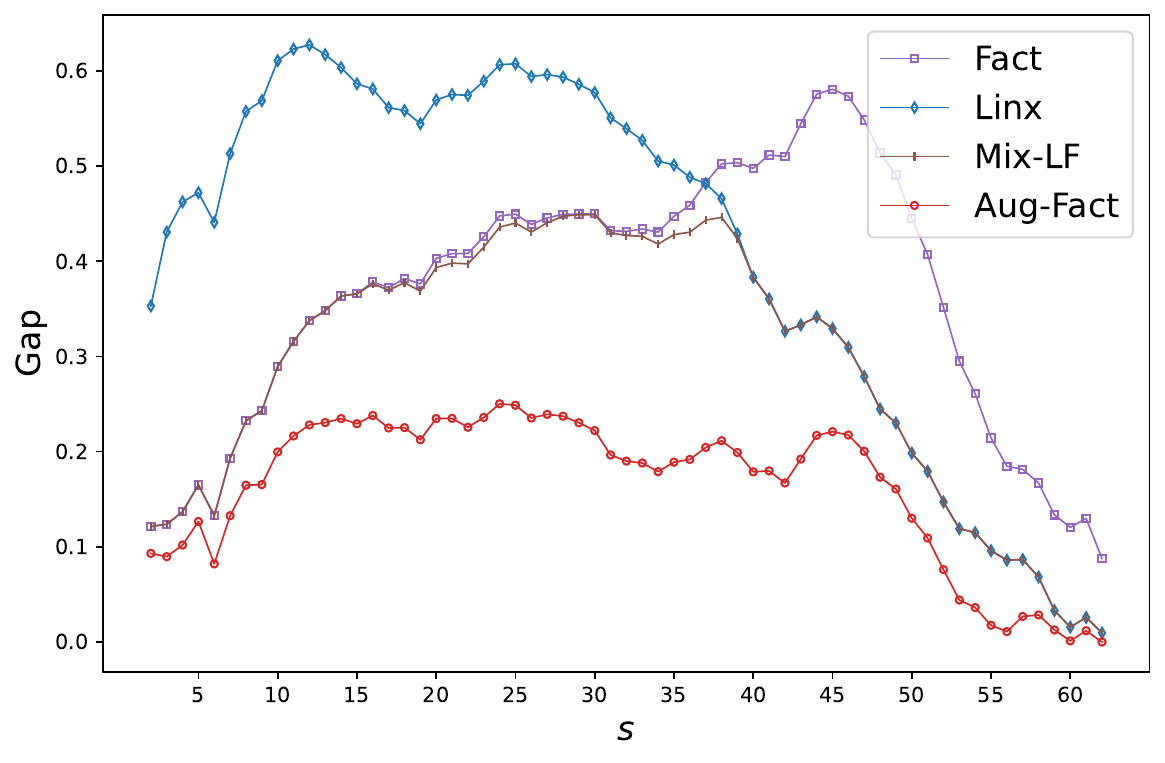}
	}
	\subfigure[Gaps] {\label{63_gap_ddf}
		\centering
\includegraphics[width=0.31\textwidth]{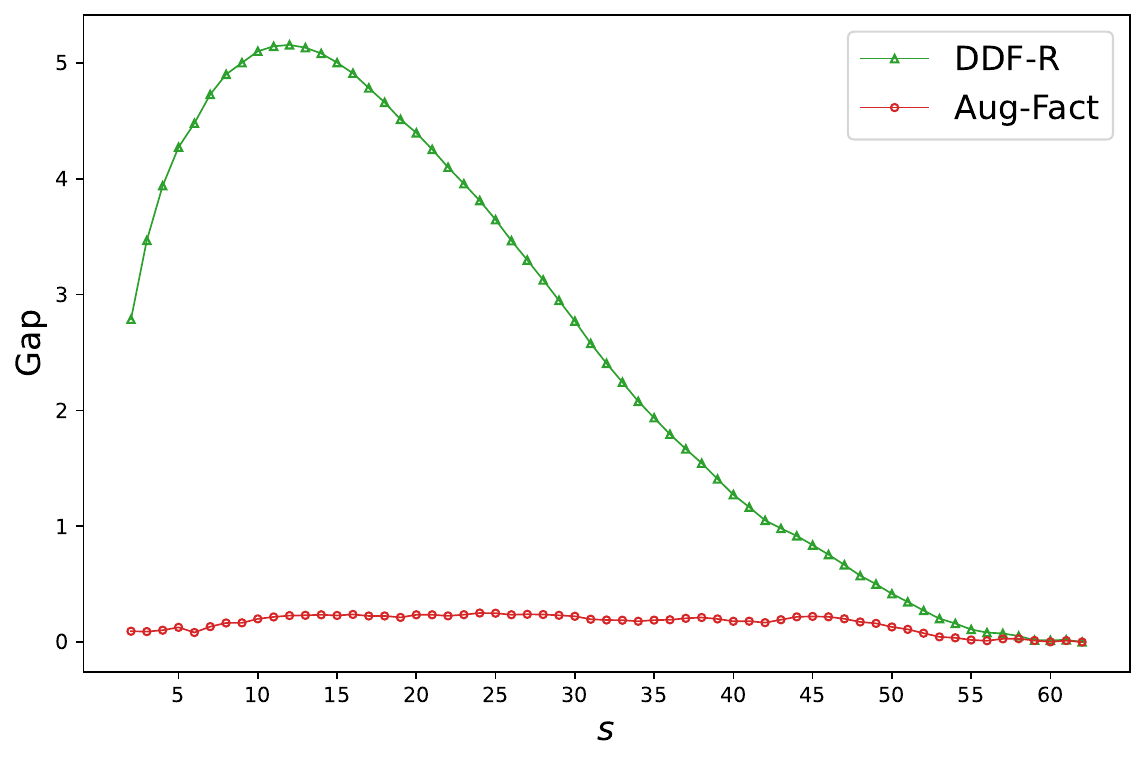}
	}
 	\subfigure[Number of variables fixed] {\label{63_fix}
		\centering
	\includegraphics[width=0.31\textwidth]{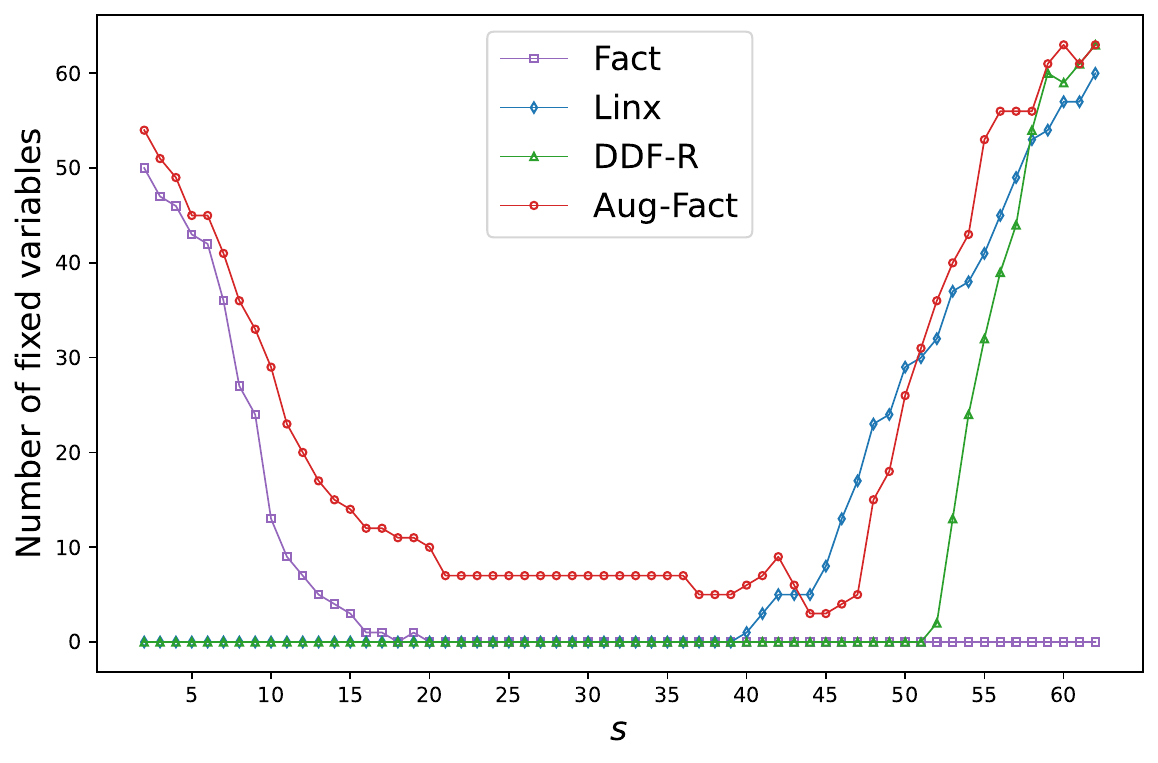}
	}
	\caption{$n=63$ with the condition number $\lambda_{\max}(\bm C)/ \lambda_{\min}(\bm C) = 48.42$ }\label{fig_comp_upper}
 \vspace{-1.5em}
\end{figure}

\begin{figure}[hbtp]
	\centering
	\hspace{-1em}
	\subfigure[Gaps] {\label{90_gap}
		\centering
  \includegraphics[width=0.32\textwidth]{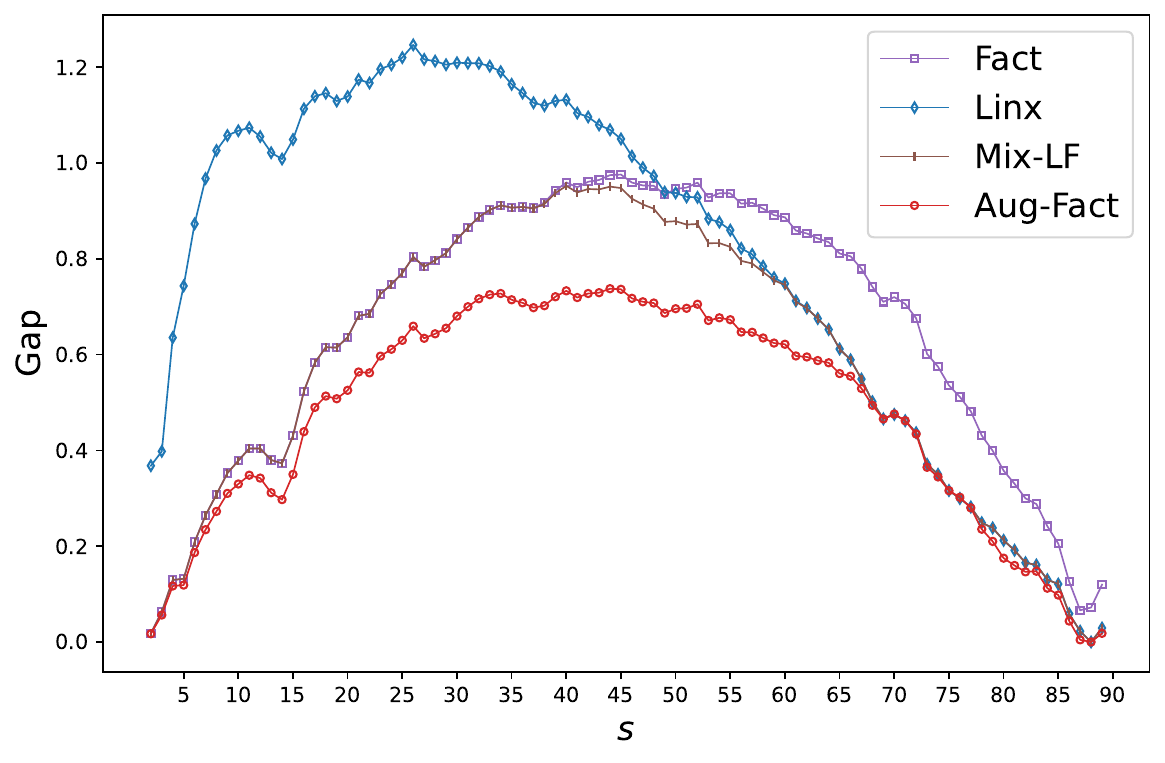}
	}
	\subfigure[Gaps] {\label{90_gap_ddf}
		\centering
  \includegraphics[width=0.32\textwidth]{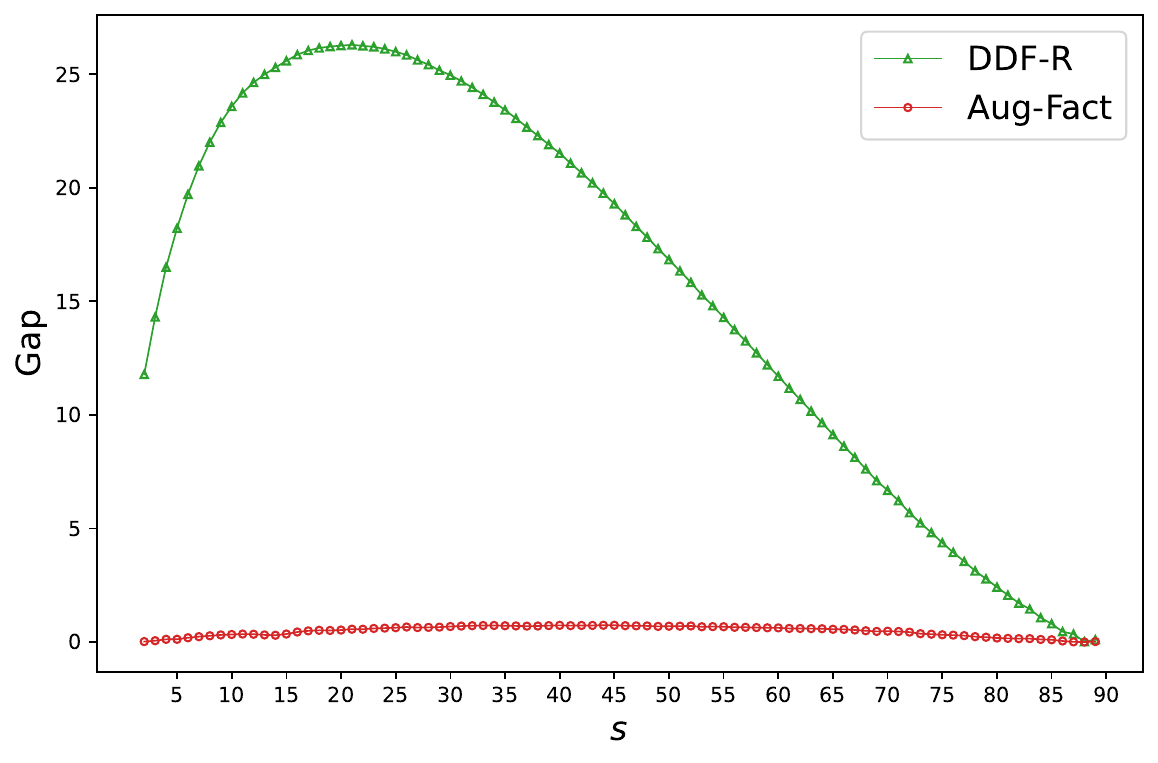}
	}
	\subfigure[Number of variables fixed] {\label{90_fix}
		\centering
	\includegraphics[width=0.32\textwidth]{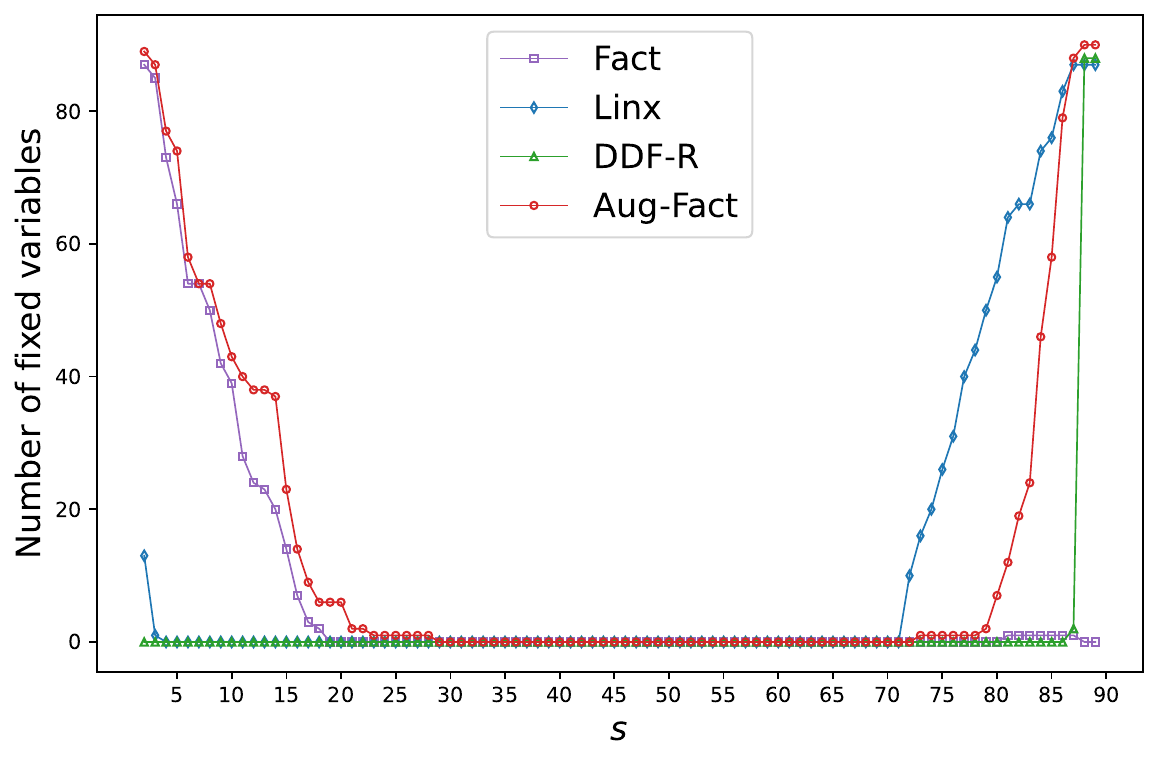}
	}
	\caption{$n=90$ with with the condition number $\lambda_{\max}(\bm C)/\lambda_{\min}(\bm C) = 200.45$  }\label{fig_comp_upper_90}
 \vspace{-1.5em}
\end{figure}

\begin{figure}[hbtp]
	\centering
 	\hspace{-1em}
	\subfigure[Gaps] {\label{124_gap}
		\centering
  \includegraphics[width=0.32\textwidth]{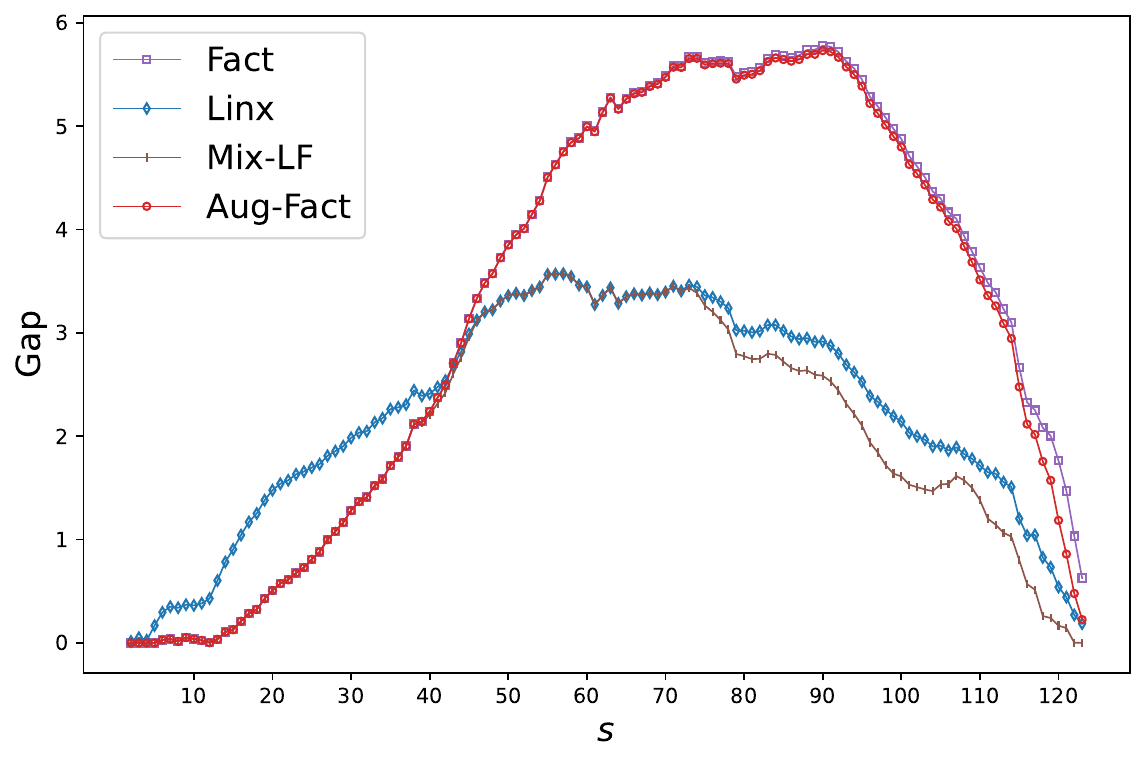}
	}
	\subfigure[Gaps] {\label{124_gap_ddf}
		\centering
  \includegraphics[width=0.32\textwidth]{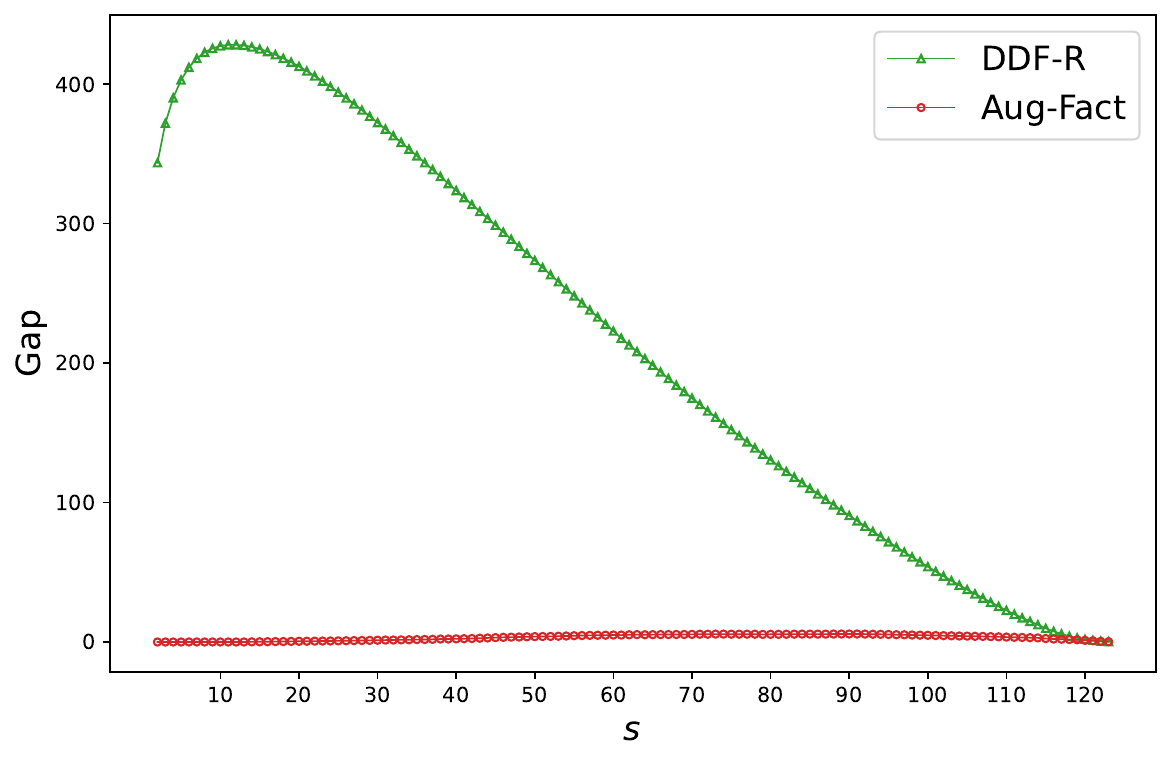}
	}
	\subfigure[Number of variables fixed] {\label{124_fix}
		\centering
	\includegraphics[width=0.32\textwidth]{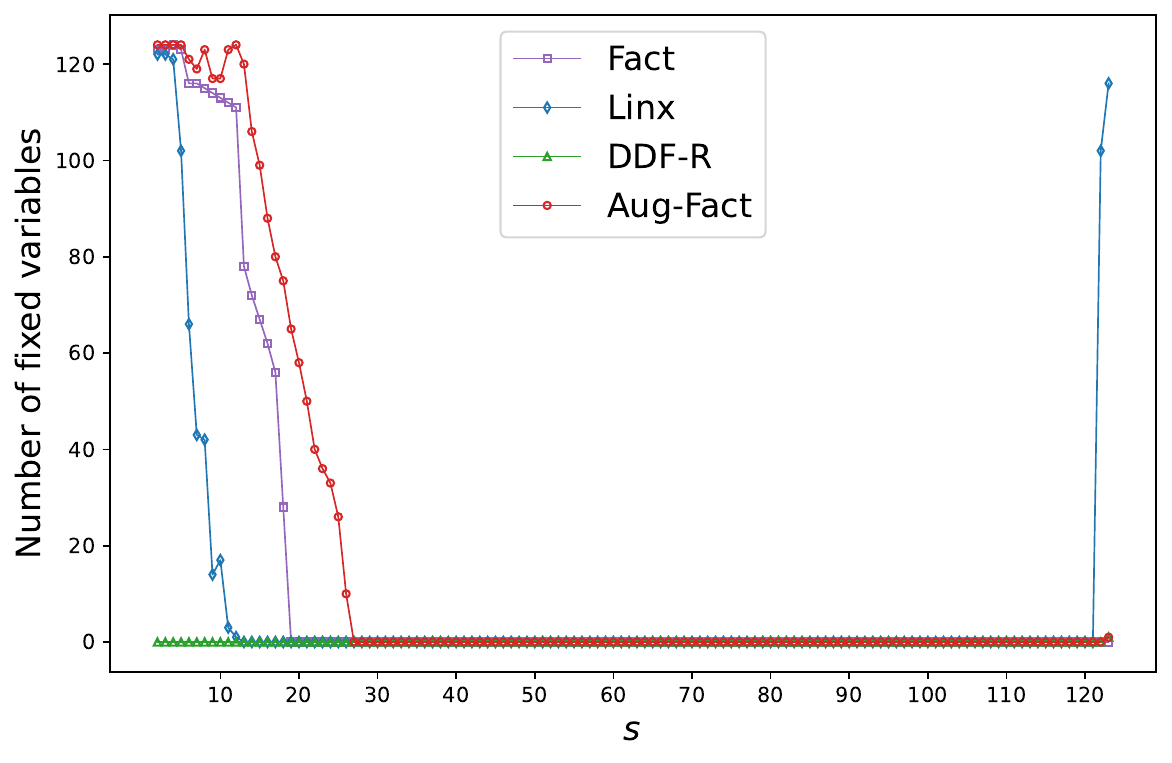}
	}
	\caption{$n=124$ with the condition number  $\lambda_{\max}(\bm C)/\lambda_{\min}(\bm C) = 78340.48$ }\label{fig_comp_upper_124}
 \vspace{-1em}
\end{figure}

\subsection{DDF: IEEE 118- and 300-bus data sets}\label{subsec:com_ddf}
This subsection tests the covariance matrices generated from two real-world IEEE data sets \citep{aminifar2009contingency}, that have been extensively applied to the phasor measurement unit (PMU) placement problem in the literature of DDF \citep{li2011phasor}.
Following the work of \cite{li2024d}, for the IEEE 118 (or 300)-bus data set, we generate two positive definite covariance matrices of $n=117$ (or $299$) based on large and small PMU standard deviations, respectively. PMU standard deviations represent different levels of measurement accuracy, leading to the covariance matrices with different condition numbers, as presented in \Cref{fig_118_large,fig_118_small,fig_300_large,fig_300_small}. Following \cite{li2024d}, for the IEEE 118-bus instance, we consider the cases  where $s\in \{10, 15, \cdots, 105\}$ to evaluate \ref{eq:upper}, and for the IEEE 300-bus instance, we set $s\in \{10, 20, \cdots, 290\}$.

First, \Cref{fig_118_large,fig_118_small,fig_300_large,fig_300_small} show that the gaps between \ref{eq:fact} and \ref{eq:ddf}, obtained from \citet[section 5]{li2024d}, are not comparable.
We also report the number of variables fixed by them using the proposed variable-fixing logic in \Cref{them:fix}. In \Cref{118_gap_fact,300_gap_large_fact,300_gap_small_fact}, \ref{eq:upper} significantly reduces the gaps of \ref{eq:fact}. However, in \Cref{118_gap_small}, the two bounds are pretty close when dealing with a vast condition number of $\bm C$. Conversely, the reduced gaps achieved by \ref{eq:upper} over \ref{eq:ddf} become most significant in this context, as seen in \Cref{118_gap_small_ddf}. In terms of the variable-fixing capacity, \ref{eq:upper} wins on nearly all test instances, with only the value of $s=25$ in \Cref{118_gap_small_fix} being an exception. 

 \vspace{-1.5em}
\begin{figure}[hbtp]
	\centering
	\hspace{-1.5em}
	\subfigure[ Gaps of \ref{eq:upper} and \ref{eq:fact}] {\label{118_gap_fact}
\includegraphics[width=0.35\textwidth]{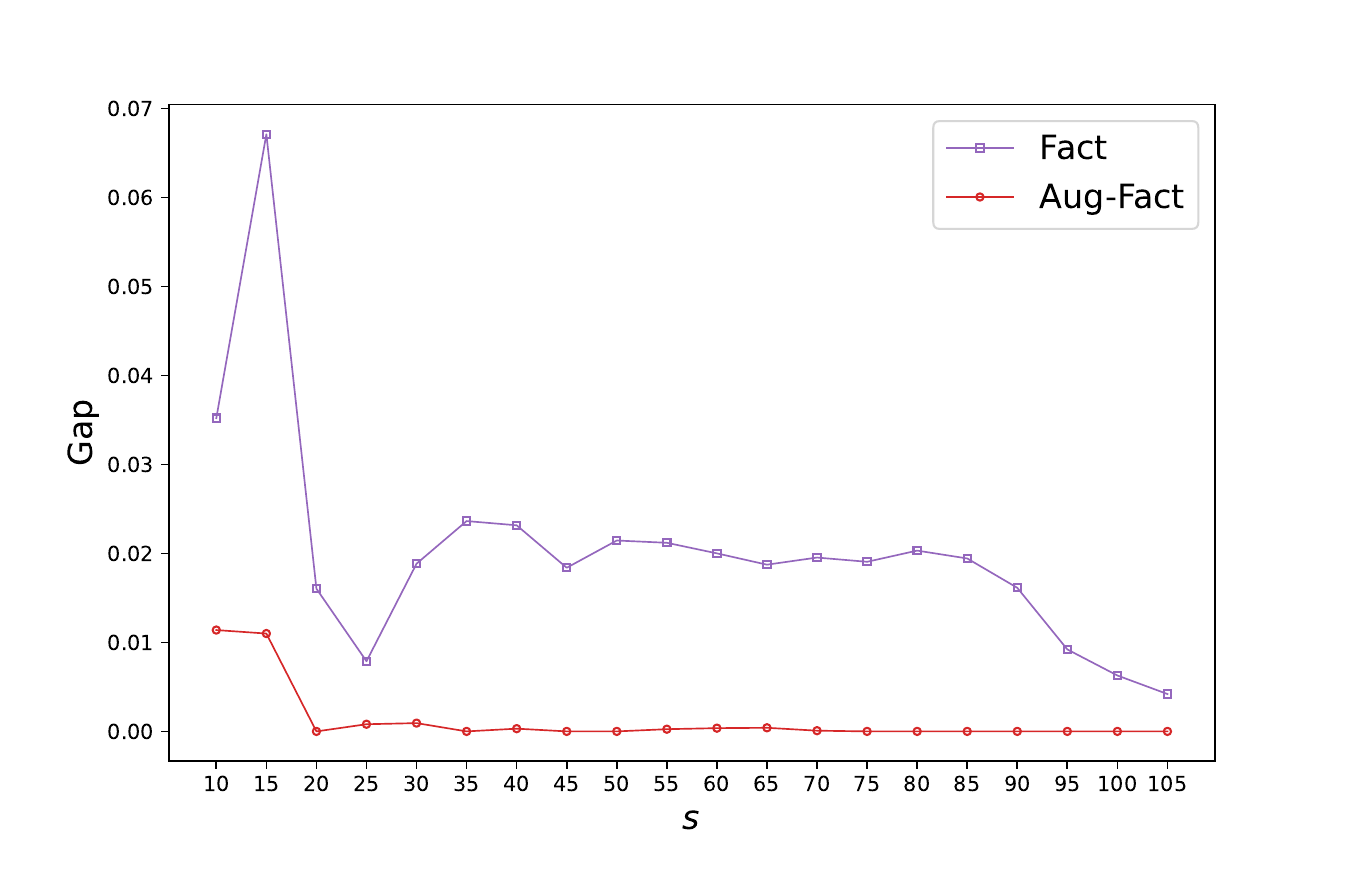}
	}
  \hspace{-1.8em}
	\subfigure[Gaps of \ref{eq:upper} and \ref{eq:ddf}] {\label{118_gap_large_ddf}
		\centering
\includegraphics[width=0.35\textwidth]{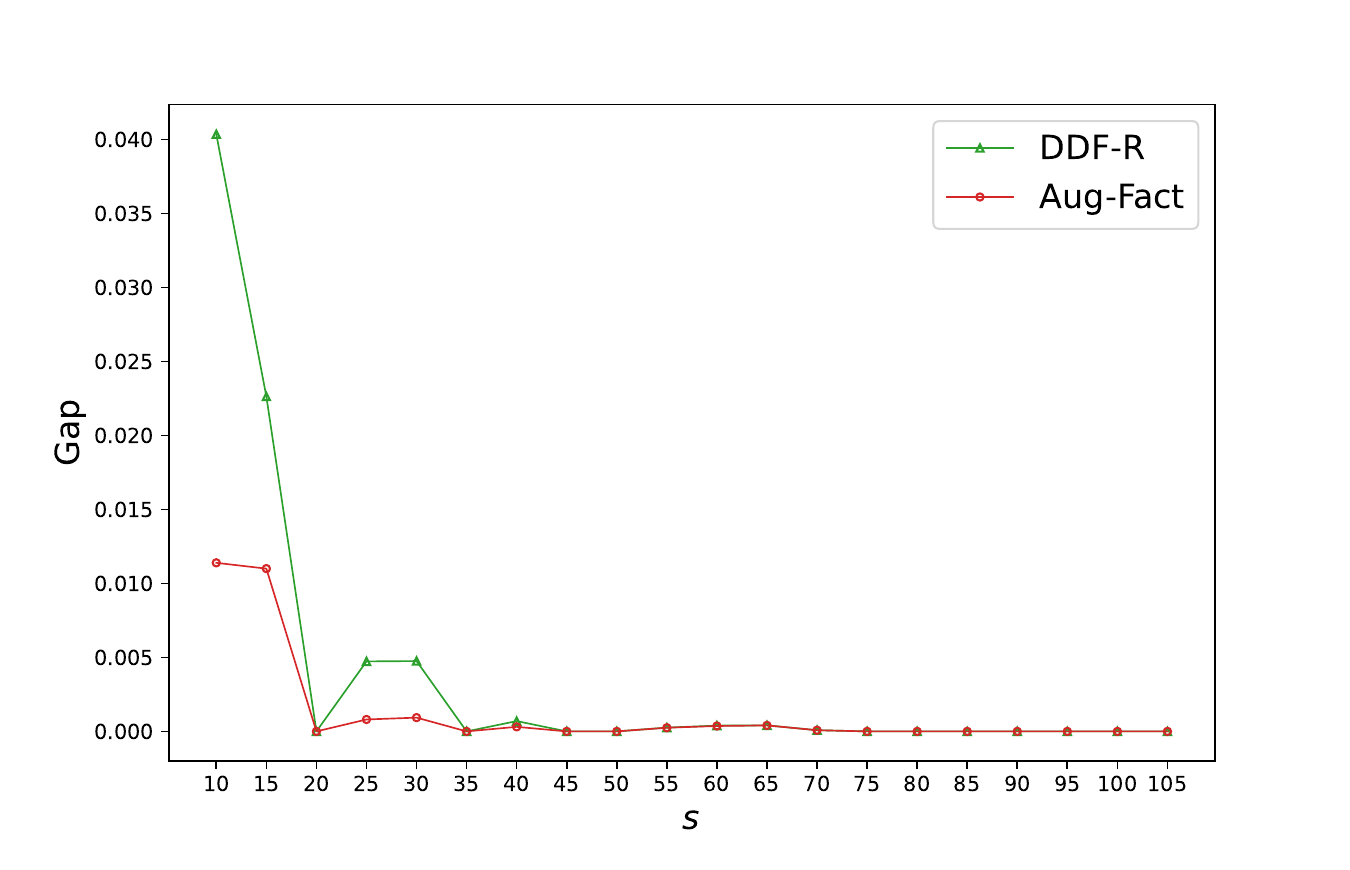}
	}
  \hspace{-1.8em}
 	\subfigure[Number of variables fixed] {\label{118_gap_large_fix}
		\centering
\includegraphics[width=0.35\textwidth]{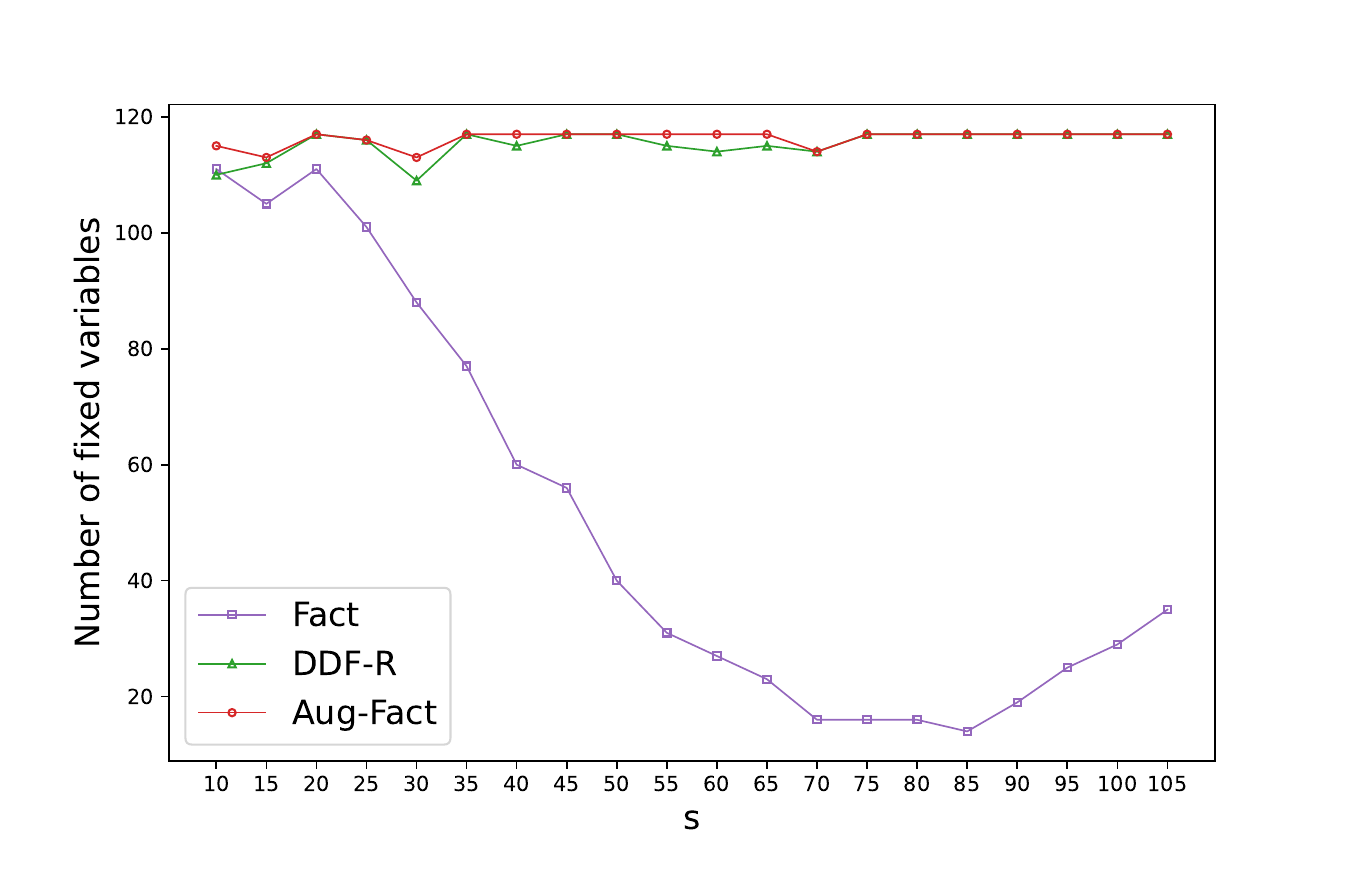}
	}
	\caption{IEEE $118$-bus instance and large PMU standard deviations with $\lambda_{\max}(\bm C)/\lambda_{\min}(\bm C)=313.27$ }\label{fig_118_large}
 	\vspace{-2em}
\end{figure}

\begin{figure}[hbtp]
	\centering
 	\hspace{-1.5em}
	\subfigure[Gaps of \ref{eq:upper} and \ref{eq:fact}] {\label{118_gap_small}
\includegraphics[width=0.35\textwidth]{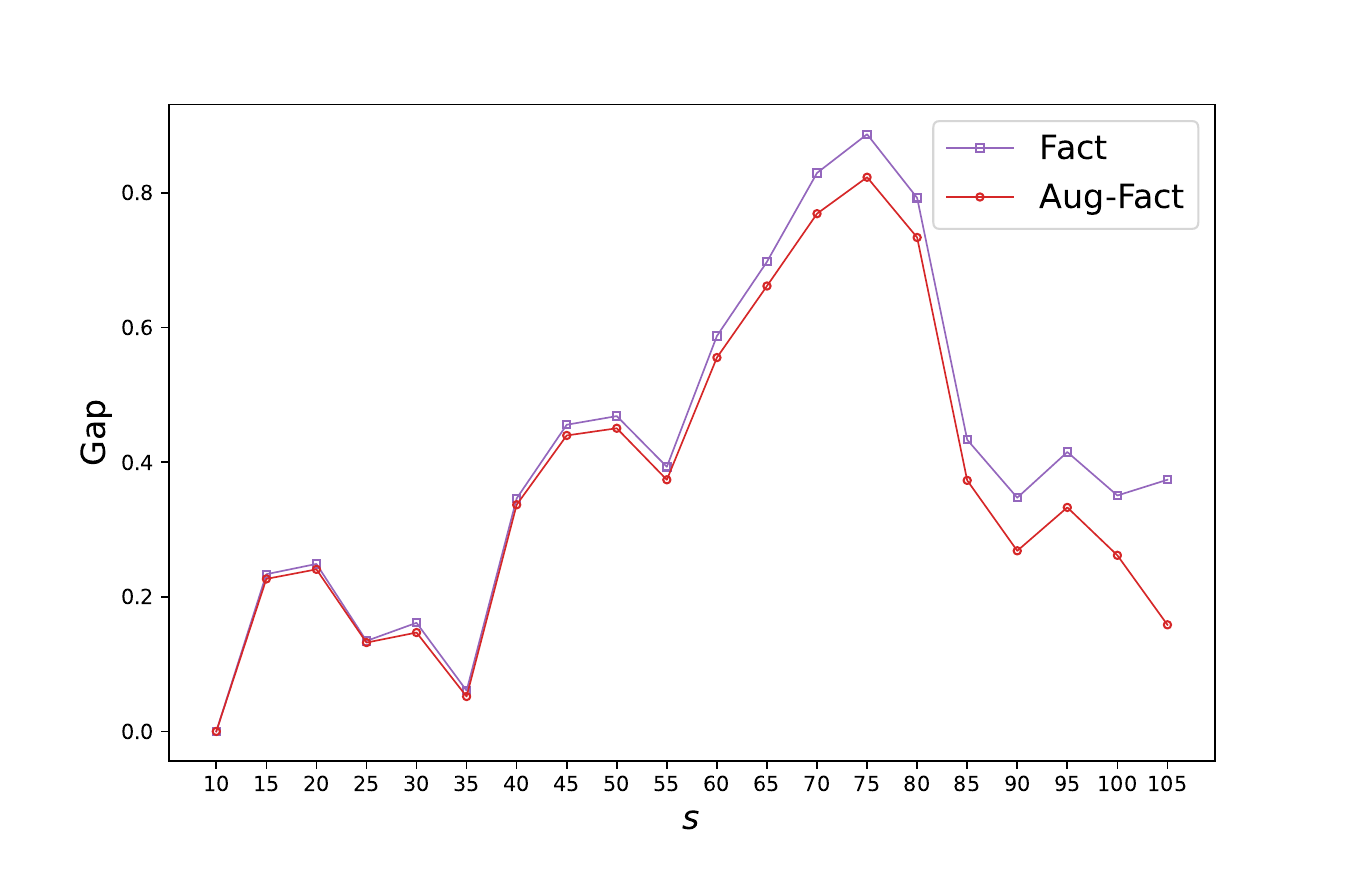}
	}
  \hspace{-1.8em}
	\subfigure[Gaps of \ref{eq:upper} and \ref{eq:ddf}] {\label{118_gap_small_ddf}
		\centering
\includegraphics[width=0.35\textwidth]{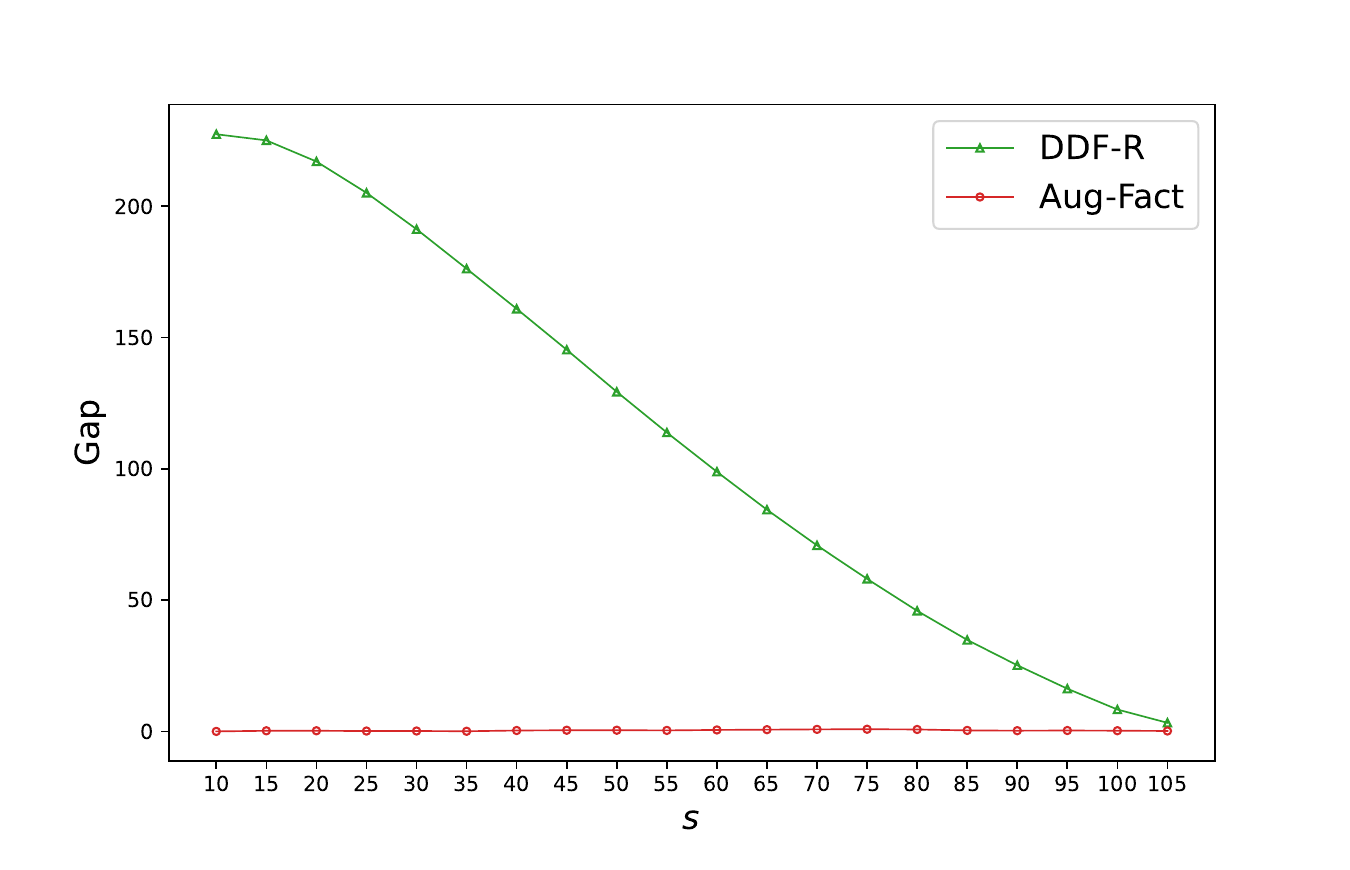}
	}
  \hspace{-1.8em}
 	\subfigure[Number of variables fixed] {\label{118_gap_small_fix}
		\centering
\includegraphics[width=0.35\textwidth]{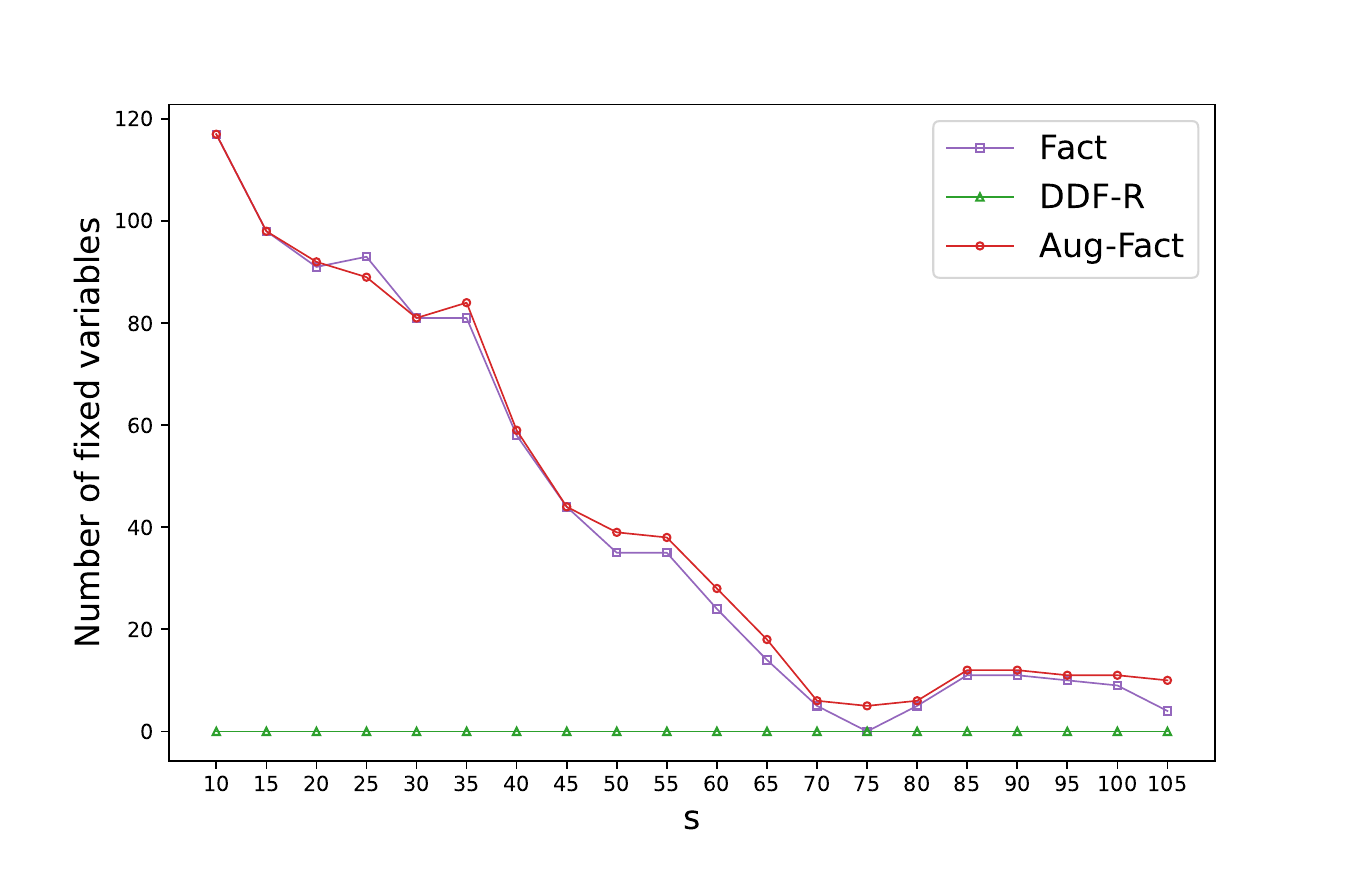}
	}
	\caption{IEEE $118$-bus instance and small PMU standard deviations with $\lambda_{\max}(\bm C)/\lambda_{\min}(\bm C)=2690744.66$ }\label{fig_118_small}
 \vspace{-2em}
\end{figure}

\begin{figure}[hbtp]
	\centering
	\hspace{-1.5em}
	\subfigure[Gaps of \ref{eq:upper} and \ref{eq:fact}] {\label{300_gap_large_fact}
\includegraphics[width=0.35\textwidth]{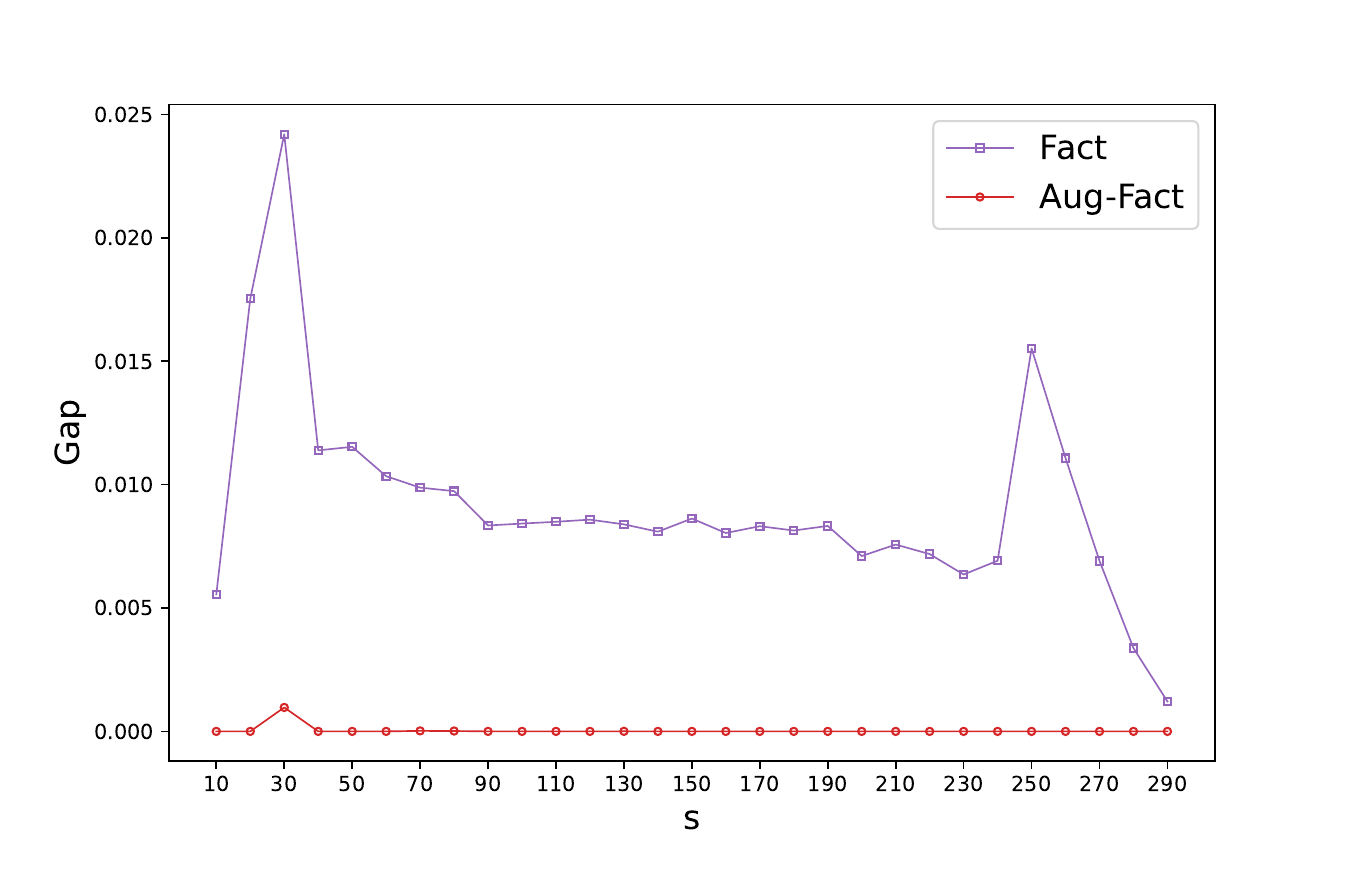}
	}
 \hspace{-1.8em}
	\subfigure[Gaps of \ref{eq:upper} and \ref{eq:ddf}] {\label{300_gap_large_ddf}
		\centering
\includegraphics[width=0.35\textwidth]{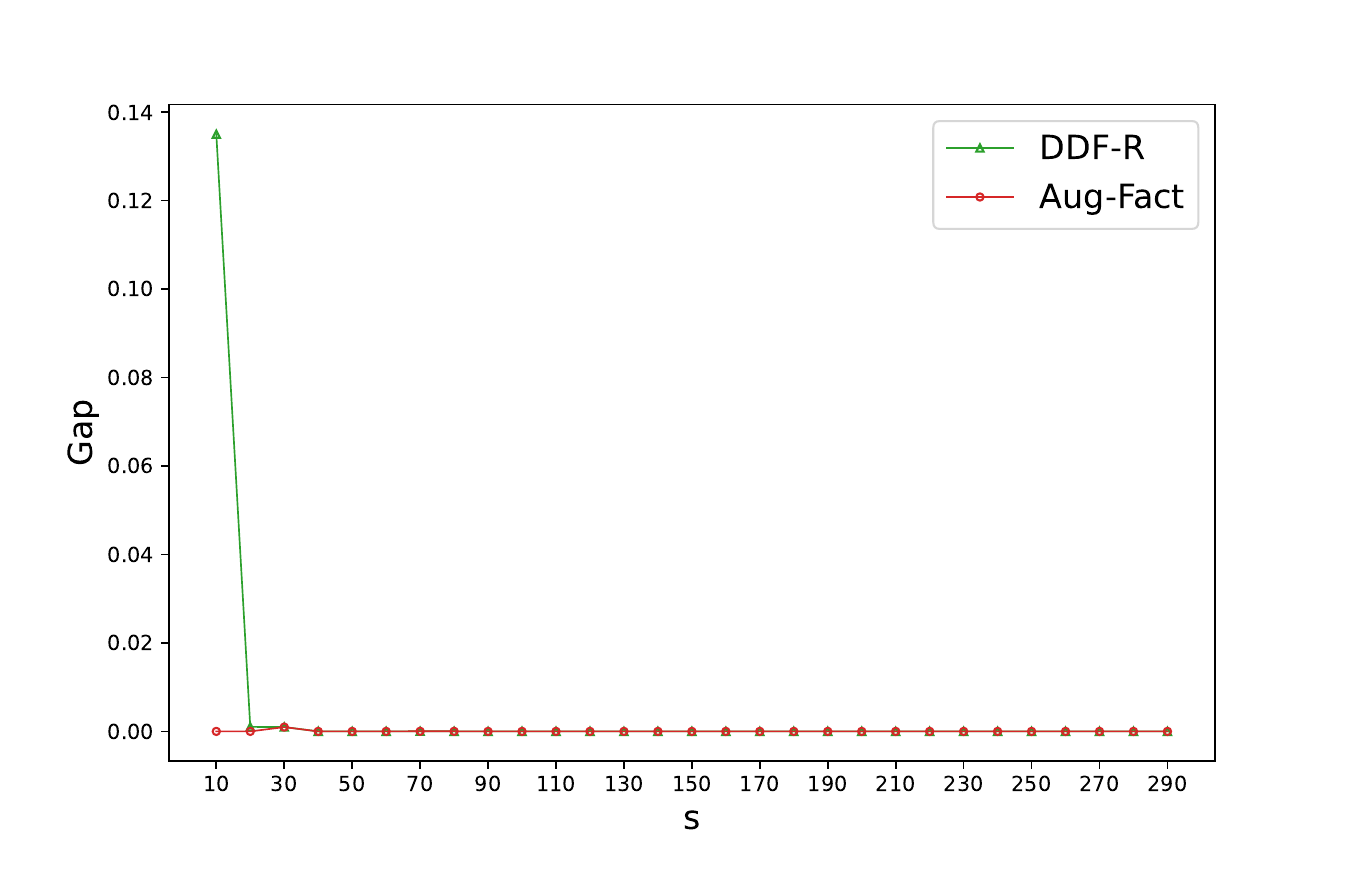}
	}
  \hspace{-1.8em}
 	\subfigure[Number of variables fixed] {\label{300_gap_large_fix}
		\centering
\includegraphics[width=0.35\textwidth]{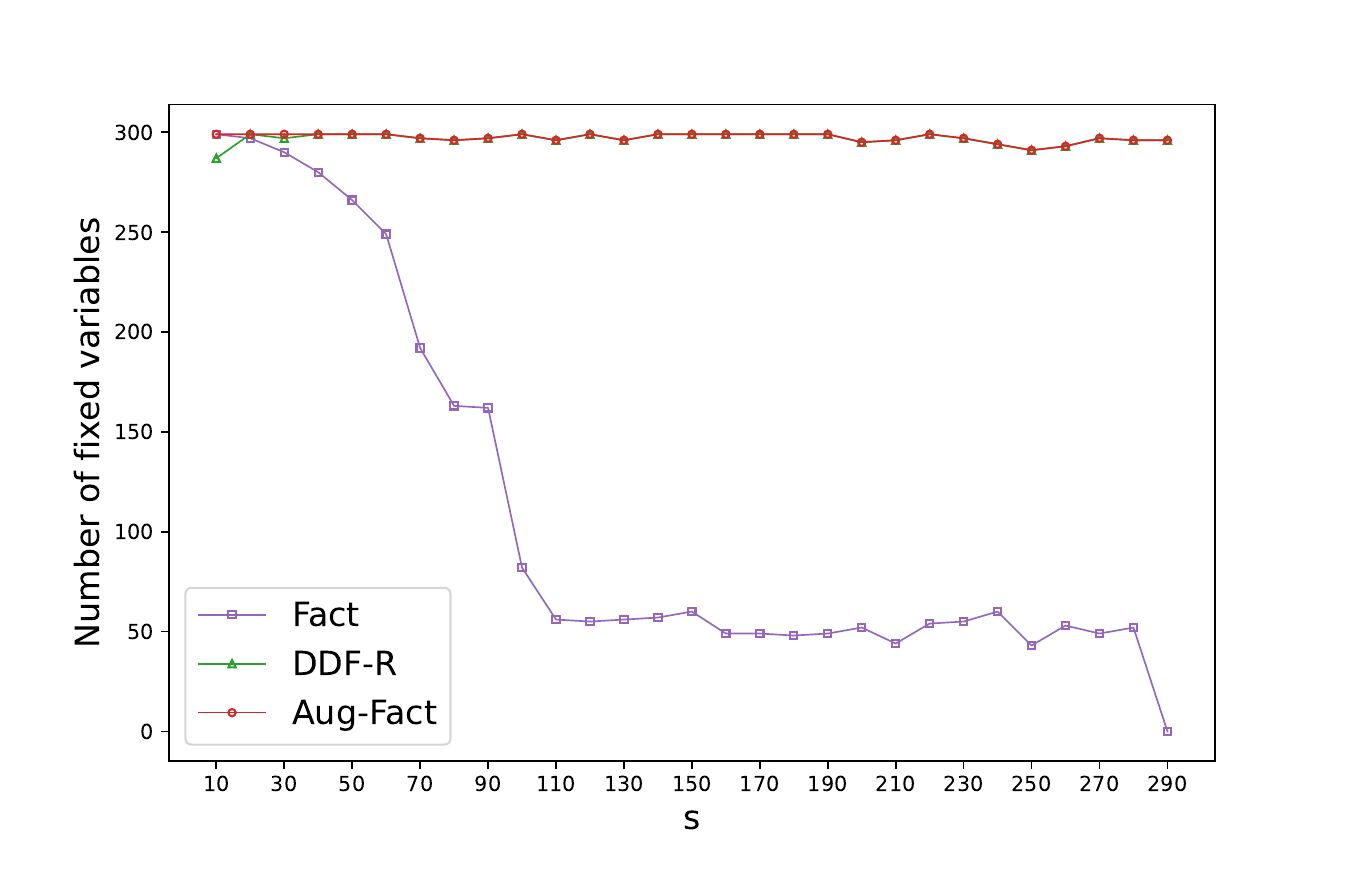}
	}
	\caption{IEEE $300$-bus instance and large PMU standard deviations with $\lambda_{\max}(\bm C)/\lambda_{\min}(\bm C)=6.50$  }\label{fig_300_large}
 	\vspace{-2em}
\end{figure}

\begin{figure}[hbtp]
	\centering
 \hspace{-1.5em}
	\subfigure[Gaps of \ref{eq:upper} and \ref{eq:fact}] {\label{300_gap_small_fact}
\includegraphics[width=0.35\textwidth]{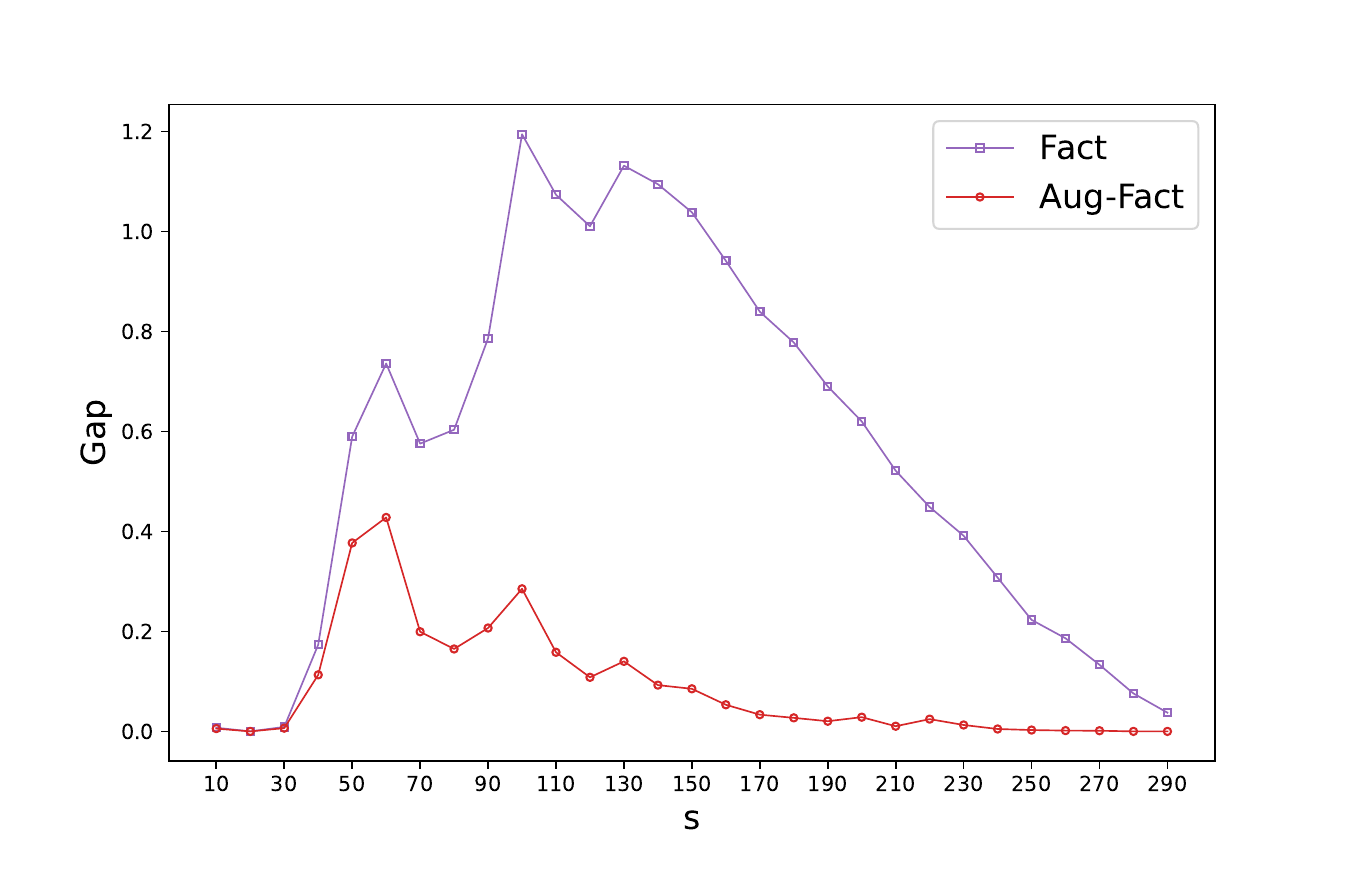}
	}
  \hspace{-1.8em}
	\subfigure[Gaps of \ref{eq:upper} and \ref{eq:ddf}] {\label{300_gap_small_ddf}
		\centering
\includegraphics[width=0.35\textwidth]{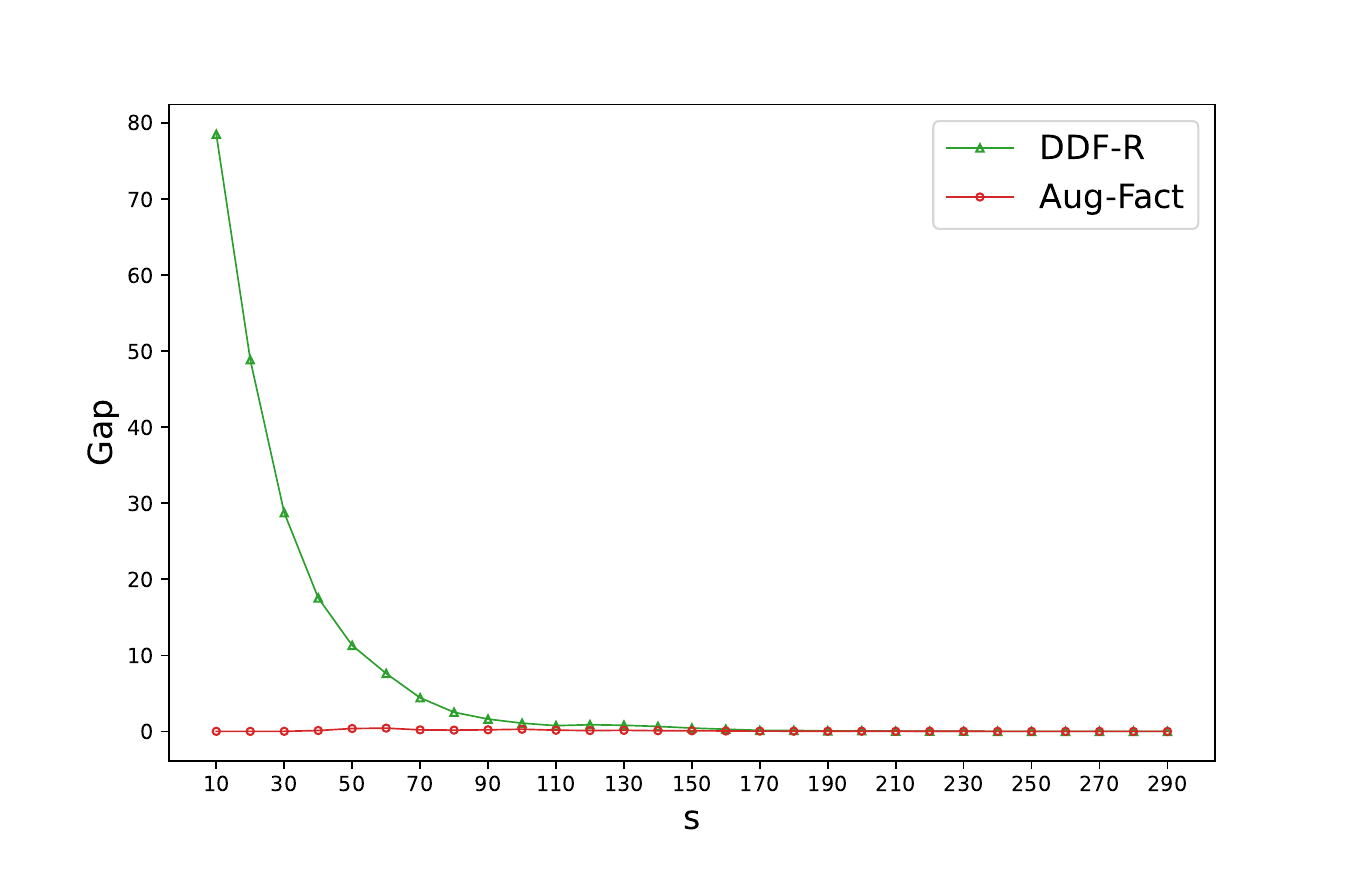}
	}
   \hspace{-1.8em}
 	\subfigure[Number of variables fixed] {\label{300_gap_small_fix}
		\centering
\includegraphics[width=0.35\textwidth]{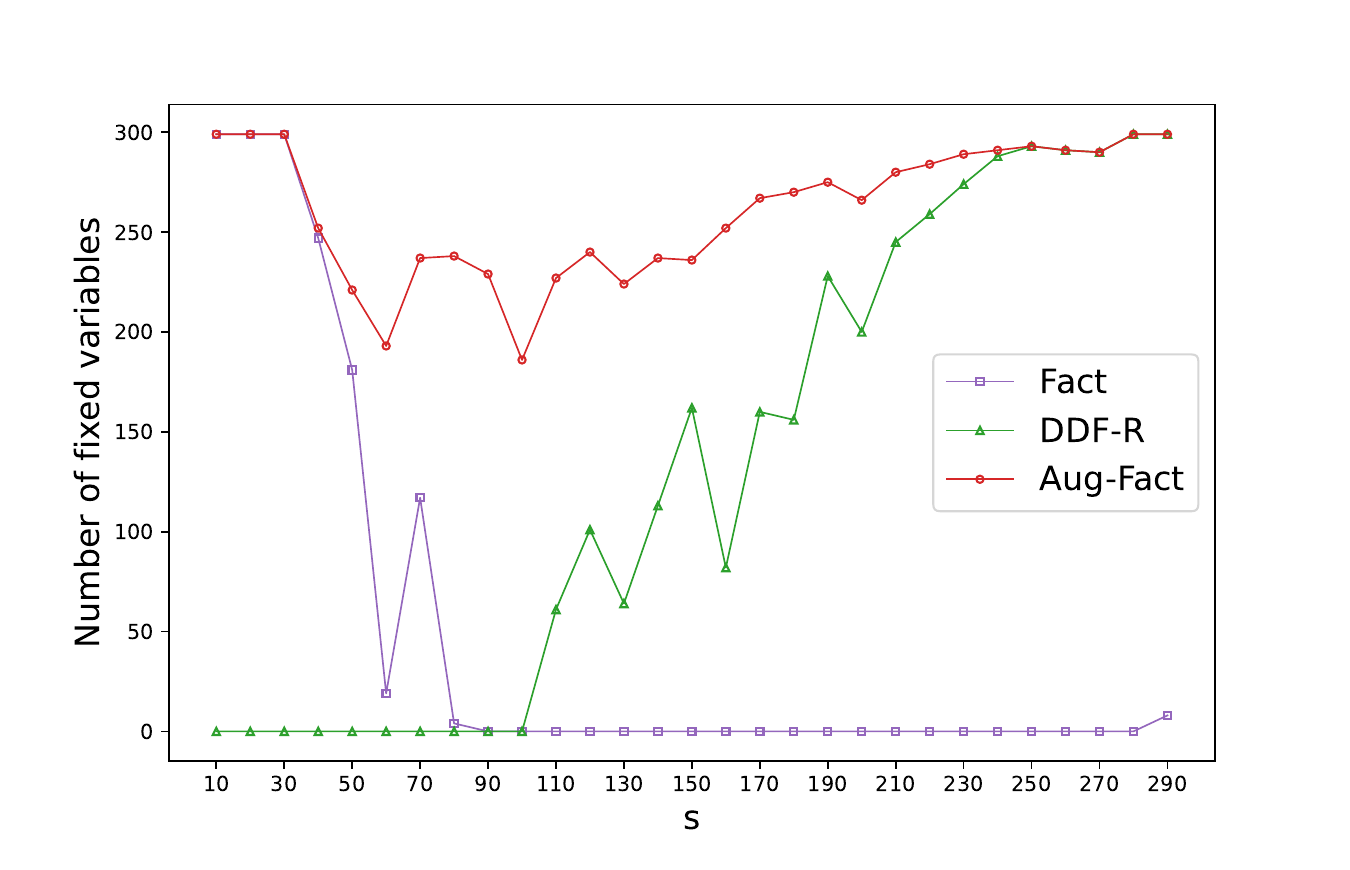}
	}
	\caption{IEEE $300$-bus instance and small PMU standard deviations with $\lambda_{\max}(\bm C)/\lambda_{\min}(\bm C)=5498.76$  }\label{fig_300_small}
 \vspace{-.5em}
\end{figure}

\section{Conclusions} \label{sec:con}
We developed a novel upper bound for the maximum entropy sampling problem, referred to as the augmented factorization bound. Our theoretical results include a thorough investigation into the monotonicity of this new bound and its superiority over two existing upper bounds, based on the theory of majorization and Schur-concave functions. Our numerical study demonstrated the strength of the proposed bound, yielding smaller gaps and fixing more variables than the state-of-the-art bounds. In future work, we plan to develop an efficient branch-and-bound implementation that incorporates the augmented factorization bound, solving MESP to optimality.
We also expect that our augmented factorization technique can apply to various machine learning and optimization problems with the cardinality constraint and Schur-concave objective functions,  such as A-optimal MESP, sparse PCA, and so on.

		\newpage
  \bibliographystyle{informs2014}
\bibliography{Areference.bib}

	\end{document}